\documentclass[]{amsart}
\usepackage{amsmath}
\usepackage{amssymb,amsfonts}
\usepackage{bbm}
\usepackage{xypic}
\usepackage{comment}
\xyoption{all}

\newtheorem{theorem}{Theorem}[section]

\newtheorem{proposition}[theorem]{Proposition}
\newtheorem{corollary}[theorem]{Corollary}
\newtheorem{lemma}[theorem]{Lemma}
\newtheorem{remark}[theorem]{Remark}
\newtheorem{example}[theorem]{Example}

\newcommand{\A}{\mathcal{A}}
\newcommand{\C}{\mathcal{C}}
\newcommand{\I}{\mathrm{I}}

\newcommand{\ZZ}{\mathcal{Z}}
\newcommand{\U}{\mathcal{U}}

\newcommand{\B}{\mathbf{B}}
\newcommand{\T}{\mathrm{T}}
\newcommand{\s}{\mathrm{S}}
\newcommand{\HH}{\mathbb{H}}
\newcommand{\DD}{\mathbb{D}}
\newcommand{\Z}{\mathbb{Z}}
\newcommand{\BB }{\mathcal{B}}
\newcommand{\M}{\mathcal{M}}
\newcommand{\N}{\mathbb{N}}

\begin{document}
\title{Higher cohomologies of commutative monoids}
\author{M. Calvo-Cervera}
\author{A.M. Cegarra}
\address{Departamento de \'Algebra,  Universidad de
Granada, 18071 Granada, Spain}
\email{acegarra@ugr.es}
\email{mariacc@ugr.es}
\thanks{This work has been supported by DGI
of Spain, Project MTM2011-22554. Also, the first author by FPU
grant FPU12-01112.}


\maketitle

\begin{abstract}
Extending Eilenberg-Mac Lane's methods, higher level cohomologies for commutative monoids are introduced and studied. Relationships with pre-existing theories (Leech, Grillet, etc.) are stated. The paper includes a cohomological classification for symmetric monoidal groupoids and explicit computations for cyclic monoids.
\end{abstract}

\section{Introduction and summary}

In \cite[Chapter X, \S 12]{maclane}, Mac Lane explains how to define, for each integer $r\geq 0$,
the $r$th level cohomology groups of a (skew) commutative DGA-algebra (differential graded augmented algebra) over a commutative ring $K$, say $D$:
Take the commutative DGA-algebra $\B^{r}\!(D)$, obtained by iterating $r$ times the  reduced bar construction on $D$, and then, for any $K$-module $A$, define
$$
H^n(D,r;A)=H^n\big(\mathrm{Hom}_K(\B^{r}\!(D),A\big), \hspace{0.4cm} n=0,1,\ldots,
$$
where $\mathrm{Hom}_K(\B^{r}\!(D),A)$ is the cochain complex obtained by applying the functor $\mathrm{Hom}_K(-,A)$ to the underlying chain complex of $K$-modules  $\B^{r}\!(D)$.

This process may be applied, for example, when $D=\Z G$ is the group ring of an abelian group $G$, regarded as a trivially graded DGA-ring,
augmented by $\alpha:\Z G\to \Z$ with $\alpha(x)=1$ for all $x\in G$. Thus, the Eilenberg-Mac Lane $r$th level cohomology groups of the abelian group $G$ with coefficients in an abelian group $A$ are defined by \begin{equation}\label{eqint1}H^n(G,r;A)=H^n(\Z G,r;A).\end{equation}
In particular, the first level cohomology groups $H^n(G,1;A)=H^n(G,A)$ are the ordinary cohomology groups of $G$ with coefficients in the trivial $G$-module $A$ \cite[Chapter IV, Corollary 5.2]{maclane}.
These $r$th level cohomology groups of abelian groups were studied primarily with interest in Algebraic Topology. For instance, they have a topological interpretation in terms of the Eilenberg-Mac Lane spaces $K(G,r)$,
owing to the isomorphisms  $H^n(G,r;A)\cong H^n\big(K(G,r),A\big)$ \cite[Theorem 20.3]{E-M-I}.
However, they early found application in solving purely algebraic problems. For example,
we could recall that central group extensions of $G$ by $A$ are classified by cohomology classes in $H^2(G,1;A)$,
while abelian group extensions of $G$ by $A$ are classified by cohomology classes in $H^3(G,2;A)$ \cite[\S 26, (26.2), (26.3)]{E-M-II};
or that second level cohomology classes in $H^4(G,2;A)$ classify braided monoidal categorical groups  \cite[Theorem 3.3]{j-s}, while third
level cohomology classes in  $H^5(G,3;A)$ classify Picard categories \cite[II, Proposition 5]{sinh}.

Here, we  introduce a generalization of Eilenberg-Mac Lane's theory for abelian groups to commutative monoids.
The obtained  $r$th level cohomology groups of a commutative monoid $M$, denoted by $$H^n(M,r;\A),$$ enjoy many desirable properties,
whose study this work and its companion paper \cite{c-c-3} are mainly dedicated to.
In our development,  the role of coefficients
is now played by abelian group objects $\A$ in the comma category of commutative monoids over $M$.
We call them {\em $\HH M$-modules} since, as a result by Grillet \cite[Chapter XII, \S 2]{grillet},
they are the same as abelian group valued functors $\A:\HH M\to \mathbf{Ab}$ on the small category $\HH M$,
whose set of objects is $M$ and set of arrows $M\times M$, with $(x,y):x\to xy$.

For any given commutative monoid $M$, the category of  chain complexes of  $\HH M$-modules is an abelian category.
In Section \ref{sect2}, we show that it is also a symmetric monoidal category, with a distributive tensor product $\A\otimes_{\HH M}\BB$, and whose unit object is $\Z$, the concentrated in degree zero complex defined by the constant $\HH M$-module given by the abelian group $\Z$ of integers. Hence, commutative {\em DGA-algebras over $\HH M$} arise as internal commutative monoids $\A$ in the symmetric monoidal category of complexes of $\HH M$-modules, endowed with a morphism of internal monoids $\A\to \Z$.

Quite similarly as for ordinary commutative DGA-algebras over a commutative ring, a reduced bar construction $\A\mapsto \B(\A)$ works on these DGA-algebras over $\HH M$. Thus, $\B(\A)$ is obtained  from $\A$ by first totalizing the double complex of $\HH M$-modules
$$
\xymatrix{\bigoplus_{p\geq 0}\A/\Z\otimes_{\HH\! M}\overset{(p\text{ factors })}\cdots
\otimes_{\HH\! M}\A/\Z},
$$
and then enriching the (suitably graded) totalized complex of $\HH M$-modules with a multiplicative structure by a shuffle product. We do this in
Section \ref{sect3}, where we also define, for any $\HH M$-module $\BB$, the $r$th level cohomology groups of $\A$ with coefficients in $\BB$ by
\begin{equation*}
H^n(\A,r;\BB)=H^n\big(\mathrm{Hom}_{\HH\! M}(\B^{r}\!(\A),\BB)\big), \hspace{0.4cm} n=0,1,\ldots \, .
\end{equation*}

Next, in Section \ref{sect4} we briefly study  free $\HH M$-modules.
These arise as a left adjoint construction to a forgetful functor from the category of $\HH M$-modules to the comma category  of sets over the underlying
set of $M$.
In particular, in Section \ref{sect5} we introduce
the free $\HH M$-module on the identity map $id_M:M\to M$, denoted by $\ZZ M$.
This becomes a (trivially graded) commutative DGA-algebra over $\HH M$ and then, for each integer positive $r$, we define
the $r$th level cohomology groups of a commutative monoid $M$ with coefficients in a $\HH M$-module $\A$ by
\begin{equation}\label{eqint2}
H^n(M,r;\A)=H^n(\ZZ M,r;\A).
\end{equation}

When $M=G$ is an abelian group, for any integer $r\geq 0$, $\B^{r}\!(\ZZ G)$ is isomorphic to the constant DGA-algebra over $\HH G$ defined by the Eilenberg-Mac Lane DGA-ring $\B^{r}\!(\Z G)$ ($=A_{N}(G,r)$ in \cite[\S 14]{E-M-I}). Hence, for any abelian group $A$, viewed as a constant $\HH G$-module, the cohomology groups $H^n(G,r;A)$ defined as in \eqref{eqint2} are naturally isomorphic to those by Eilenberg and Mac Lane in \eqref{eqint1}, which, recall, compute the cohomology groups of the spaces $K(G,r)$ as $H^n(G,r;A)\cong H^n\big(K(G,r),A\big)$. In the companion paper \cite{c-c-3} we show that, for any commutative monoid $M$, there are isomorphisms $$H^n(M,r;\A) \cong H^n(\overline{W}^{\,r}\! M,\A),$$  where $H^n(\overline{W}^{\,r}\! M,\A)$, $n\geq 0$, are Gabriel-Zisman cohomology groups \cite[Appendix II]{G-Z} of the underlying simplicial set of the simplicial monoid $\overline{W}^{\,r}\! M$, obtained by iterating the $\overline{W}$ construction on the constant simplicial monoid defined by $M$.

An analysis of the complex $\B(\ZZ M)$, for $M$ any commutative monoid,
leads us in Proposition \ref{h1l} to identify the cohomology groups $H^n(M,1;\A)$ with the standard cohomology groups  $ H^n_{^\mathrm{L}}(M,\A)$  by Leech \cite{leech}. Recall that Leech cohomology groups of a (not necessarily commutative) monoid $M$ take coefficients in abelian group valued functors on the category $\DD M$, whose objects are the elements of $M$ and arrows triples $(x,y,z):y\to xyz$. When the monoid $M$ is commutative, there is a natural functor $\DD M\to \HH M$ which is the identity on objects and carries a morphism $(x,y,z)$ of $\DD M$ to the morphism $(y,xz)$ of $\HH M$. Via this functor, every $\HH M$-module $\A$ is regarded as a $\DD M$-module and we prove that,
for any commutative monoid $M$ and $\HH M$-module $\A$, there are natural isomorphisms
\begin{equation*}\label{comleech}H^n(M,1;\A)\cong H^n_{^{_\mathrm{L}}}(M,\A),  \hspace{0.4cm} n=0,1,\ldots \, .\end{equation*}

For any $r\geq 2$, we show  explicit descriptions of the complexes $\B^{r}\!(\ZZ M)$
truncated at dimensions $\leq r+3$, which are useful both for theoretical and computational interests concerning the cohomology groups $H^n(M,r;\A)$ for $n\leq r+2$. Some conclusions here summarize  as follows:
\begin{itemize}
\item $H^0(M,r;\A)\cong H^0(M,1;\A)\cong H^0_{^\mathrm{L}}(M,\A)\cong \A(e)$,
\end{itemize}
where $\A(e)$ is the abelian group attached by $\A$ at the identity $e$ of the monoid.
\begin{itemize}
\item $H^n(M,r;\A)= 0$,  for  $0<n<r$,
\vspace{0.2cm}
\item $H^r(M,r;\A)\cong H^1(M,1;\A)\cong H^1_{^\mathrm{L}}(M,\A)  \cong H^1_{^{_\mathrm{ G}}}(M,\A)$,

\vspace{0.2cm}
\item $H^{r+1}(M,r;\A)\cong H^{3}(M,2;\A)\cong H^2_{^{_\mathrm{ G}}}(M,\A)$.
\end{itemize}
where $H^n_{^\mathrm{G}}(M,\A)$ denotes the $n$-th cohomology group by Grillet \cite{grillet91, grillet}.

\begin{itemize}
\item $H^{4}(M,2;\A)\cong  H^3_{^{_\mathrm{C}}}(M,\A)$,
\end{itemize}
where $H^3_{^{_\mathrm{C}}}(M,\A)$ is the  third commutative cohomology group by the authors in \cite{c-c-1}.
\begin{itemize}
\vspace{0.1cm}
\item $H^{r+2}(M,r;\A)\cong H^5(M,3;\A)$,  for  $r\geq 3$.
\end{itemize}
\begin{itemize}
\item There are natural inclusions $H^3_{^{_\mathrm{G}}}(M,\A)\subseteq H^5(M,3;\A)\subseteq H^3_{^{_\mathrm{C}}}(M,\A)$.
\end{itemize}

Most of these cohomology groups above have known  algebraic interpretations. For example,
elements of $H^1(M,1;\A)=H^1_{^{_\mathrm{ L}}}(M,\A)$ are {\em derivations} \cite[Chapter II, 2.7]{leech}.
Cohomology classes in  $H^2(M,1;\A)=H^2_{^{_\mathrm{L}}}(M,\A)$ are isomorphism classes of {\em group coextensions}
\cite[Chapter V, \S2]{leech} (or \cite[Theorem 2]{wells}), while elements of $H^3(M,2;\A)=H^2_{^\mathrm{G}}(M,\A)$ classify {\em abelian group coextensions} \cite[Chapter V, \S 4]{grillet}.
Cohomology classes in $H^3(M,1;\A)=H^3_{^{_\mathrm{ L}}}(M,\A)$ are equivalence classes of  {\em monoidal abelian groupoids} \cite[Theorem 4.3]{c-c-h},
elements of $H^4(M,2;\A)=H^3_{^{_\mathrm{C}}}(M,\A)$ are equivalence classes of {\em braided monoidal abelian groupoids} \cite[Theorem 4.5]{c-c-1}, and elements of $H^3_{^{_\mathrm{G}}}(M,\A)$ are equivalence classes of {\em strictly commutative monoidal abelian groupoids} \cite[Theorem 3.1]{c-c-h2}.
Thus, among them, only the cohomology groups $H^5(M,3;\A)$ are pending of interpretation, and we solve this in Section \ref{sect6}.
Here we give a natural interpretation of the cohomology classes in $H^5(M,3;\A)$ in terms of equivalence
classes of {\em symmetric monoidal abelian groupoids}, that is, groupoids $\M$,  whose isotropy groups $\mathrm{Aut}_\M(x)$
are all abelian, endowed with a monoidal structure by a tensor functor $\otimes:\M\times \M\to \M$,
a unit object $\I$, and coherent associativity, unit and commutativity constraints
$\boldsymbol{a}:(x\otimes y)\otimes z \cong x\otimes(y\otimes z)$,
 $\boldsymbol{l}:\I \otimes x \cong x$, and $\boldsymbol{c}:x\otimes
y\cong y\otimes x$ which satisfy the symmetry condition $\boldsymbol{c}^2=id$. The classification of symmetric monoidal abelian groupoids we give extends that, above refereed, by Sinh  in \cite[II, Proposition 5]{sinh} for Picard categories.

In last Section \ref{sect7}, we compute the cohomology groups $H^n(M,r;\A)$, for $n\leq r+2$, when $M$ is any cyclic monoid.

\section{Commutative differential graded algebras over $\HH M$}\label{sect2}
Throughout this paper $M$ denotes a {\em commutative multiplicative
 monoid}, whose unit is $e$.

 As noted in the introduction, in \cite[Chapter XII, \S 2]{grillet} Grillet observes that the category of abelian group objects in the slice category of
commutative monoids over $M$, $\mathbf{CMon}\!\downarrow\!\!_M$, is
equivalent to the category of abelian group valued functors
 on the small category $\HH M$, whose  object set is $M$ and arrow set $M\times M$, where $(x,y):x\to xy$. Composition is given by $(xy,z)(x,y)=(x,yz)$, and
the identity of an object $x$ is $(x,e)$. The category of functors
from $\HH M $ into the category of abelian groups will be denoted by
$$
\HH M\text{-}\mathrm{Mod}
$$
and called the category of {\em $\HH M $-modules}. An $\HH M $-module, say  $ \A :\HH M \to \mathbf{Ab}$, then consists of abelian groups $ \A(x)$, one for each $x\in M$, and homomorphisms
$y_*: \A(x)\to  \A(xy)$, one for each $x,y\in M$, such that, for any $x,y,z\in M$,
$y_*z_*=(yz)_*:  \A(x)\to  \A(xyz)$ and, for any $x\in M$,
$e_*=id_{ \A(x)}: \A(x)\to \A(xe)= \A(x)$.

 For instance, let \begin{equation}\label{z}\Z:\HH M\to \mathbf{Ab}, \hspace{0.4cm}x\mapsto \Z(x)=\Z\{ x\},\end{equation} be the  $\HH M$-module
 which associates to each element $x\in M$ the free abelian
group on the generator $x$, and to each pair $(x,y)$ the isomorphism of abelian groups
$y_*: \Z(x)\to \Z( xy)$
given on the generator by $y_*x=xy$. This is
isomorphic to the  $\HH M$-module defined by the constant functor
on $\HH M$ which associates the abelian group of integers $\Z$ to any $x\in M$.

For two $\HH M $-modules $ \A $ and $ \BB $, a morphism between them (i.e., a natural transformation) $f: \A \to  \BB $ consists of homomorphisms $f_x: \A(x)\to  \BB(x)$, such that, for any $x,y\in M$, the square below commutes.
$$
\xymatrix{ \A(x)\ar[r]^{f_x}\ar[d]_{y_*}& \BB(x)\ar[d]^{y_*}\\  \A(xy)\ar[r]^{f_{xy}}& \BB(xy)}
$$

The category of  $\HH\!M$-modules is abelian and we refer to \cite[Chapter  IX, \S 3]{maclane} for details.
Recall that the set of morphisms between two  $\HH\!M$-modules $ \A $ and $ \BB $, denoted by $\mathrm{Hom}_{\HH M}( \A , \BB )$,
is an abelian group by objectwise addition, that is, if $f,g: \A \to  \BB $ are morphisms, then  $f+g: \A \to  \BB $
is defined by setting $(f+g)_x=f_x+g_x$, for each $x\in M$. The zero  $\HH M$-module is the constant functor
$0:\HH M\to \mathbf{Ab}$ defined by the trivial abelian group $0$, and the direct sum of two $\HH M$-modules $ \A $ and $ \BB $ is given by taking direct sum at each object, that is, $( \A \oplus \BB )(x)= \A(x)\oplus  \BB(x)$. Similarly, all limits and colimits (in
particular, kernels, images, cokernels, etc. ) in the category $\HH M\text{-}\mathrm{Mod}$
 are pointwise constructed.

\begin{remark}\label{eqabhhm}{\em Every abelian group $A$ defines a {\em constant} $\HH M$-module, equally denoted by $A$, such
that $A(x)=A$ and $y_*=id_A:A(x)\to A(xy)$, for any $x,y\in M$. In this way, the category of abelian groups becomes a full subcategory of the category of $\HH M$-modules.

When $M=G$ is an abelian group, then this inclusion $\mathbf{Ab}\hookrightarrow \HH G\text{-Mod}$ is actually an equivalence of categories. In the other direction, we have the functor associating to each $\HH G$-module $\A$ the abelian group $\A(e)$, and there is natural isomorphism of $\HH G$-modules $\A\cong \A(e)$ whose component at each $x\in G$ is the isomorphism of abelian groups $x_*^{-1}:\A(x)\to \A(e)$.}
\end{remark}

\subsection{Tensor product of $\HH M$-modules.}  For any two $\HH M$-modules $ \A$ , $\BB $, their {\em tensor product}, denoted
by $ \A \otimes_{\HH M} \BB $, is the $\HH M$-module defined as
follows: It attaches to any $x\in M$ the abelian group defined by
the coequalizer sequence of homomorphisms
\begin{equation*}
\xymatrix@C=22pt{\bigoplus\limits_{uvw=x}\Z( u)\otimes \A(v)\otimes  \BB(w)\ar@<2pt>[r]^-\phi
\ar@<-2pt>[r]_-\psi&\bigoplus\limits_{zt=x} \A(z)\otimes \BB(t)\ar@{->>}[r]&
( \A \otimes_{\HH M} \BB )(x),
}
\end{equation*}
where, for any two abelian groups $A$ and $B$,  $A\otimes B$ denotes their tensor product as $\Z$-modules, the direct sum on the left is taken over all triples $(u,v,w)\in M^3$ such that $uvw=x$, the direct sum on the middle is over all pairs $(z,t)\in M^2$ with $zt=x$, and the homomorphisms $\phi$ and $\psi$ are defined by
$$\begin{array}{ll} \phi\big(u\otimes a_v\otimes b_w\big)=u_*a_v\otimes b_w\in  \A(uv)\otimes \BB(w), \\[4pt]
\psi\big(u\otimes a_v\otimes b_w\big)=a_v\otimes u_*b_w\in  \A(v)\otimes \BB(uw),\end{array}$$
for all $u,v,w\in M$ with $uvw=x$, $a_v\in  \A(v)$, and $b_w\in  \BB(w)$. For any pair $(x,y)\in M^2$, the  homomorphism $$y_*:( \A \otimes_{\HH M} \BB )(x)\to
( \A \otimes_{\HH M} \BB )(xy)$$ is given on generators by
$$
y_*\big(a_z\otimes b_t\big)=y_*a_z\otimes b_t=a_z\otimes y_* b_t ,\hspace{0.6cm}  \big(a_z\in \A(z),  \, b_t\in  \BB(t), zt=x\big).
$$

If $f: \A \to  \A'$ and $g: \BB \to  \BB'$ are morphisms
of $\HH M$-modules, then there is an induced one $f\otimes
g: \A \otimes_{\HH M} \BB \to  \A'\otimes_{\HH M} \BB'$ such that, for each $x\in M$, the homomorphism
$$(f\otimes
g)_x:( \A \otimes_{\HH M} \BB )(x)\to ( \A '\otimes_{\HH M} \BB ')(x)$$ is given on generators by
$$(f\otimes g)_{x}\big(a_z\otimes b_t\big)=f_za_z\otimes g_tb_t,\hspace{0.6cm}  \big(a_z\in \A(z),  \, b_t\in  \BB(t), zt=x\big).$$
Thus, we have a distributive tensor functor
\begin{equation*}
-\otimes_{\HH M}-:\HH M\text{-}\mathrm{Mod} \times \HH M\text{-}\mathrm{Mod}
\to \HH M\text{-}\mathrm{Mod}.
\end{equation*}
Further,  there are canonical isomorphisms of $\HH M$-modules
\begin{equation*}\begin{array}{l}\boldsymbol{l}_\A:\Z\otimes_{\HH M} \A \cong  \A ,
\hspace{0.4cm}
\boldsymbol{c}_{\A,\BB}: \A \otimes_{\HH M} \BB \cong  \BB \otimes_{\HH M}  \A , \\
\boldsymbol{a}_{\A,\BB,\C}:  \A \otimes_{\HH M} ( \BB \otimes_{\HH M} \C)\cong
( \A \otimes_{\HH M} \BB )\otimes_{\HH M}\C,\end{array}
\end{equation*}
 respectively defined by the formulas
$$\begin{array}{l}\boldsymbol{l}_{zt}(z\otimes a_t)=z_*a_t,\hspace{0.4cm}
\boldsymbol{c}_{zt}(a_z\otimes
b_t)=b_t\otimes a_z,\\
\boldsymbol{a}_{yzt}( a_y\otimes(b_z\otimes c_t))=(a_y\otimes
b_z)\otimes c_t,\end{array}$$ which are easily proven to be natural and
coherent in the sense of \cite[Theorem 5.1]{Mac2}. Therefore, $\HH M\text{-}\mathrm{Mod}$  is a symmetric
monoidal category. We will usually treat the constraints above as
identities, so we think of $\HH M\text{-}\mathrm{Mod}$ as a
symmetric strict monoidal category.

\subsection{Tensor product of complexes of $\HH M$-modules.}
The  (positive) complexes  of $\HH M$-modules
$$\A=\cdots \to \A_2\overset{\partial}\to \A_1\overset{\partial}\to \A_0$$
and the morphisms between them also form an abelian symmetric monoidal category, where the tensor product
 $\A\otimes_{\HH M}\BB$ of two complexes of $\HH M$-modules $\A$ and $\BB$ is the graded $\HH M$-module
whose component of degree $n$ is
$$
\xymatrix{(\A\otimes_{\HH M}\BB)_n=\bigoplus\limits_{p+q=n}\A_p
\otimes_{\HH M}\BB_q,}
$$ and whose differential $\partial^\otimes$, at any $x\in M$, $$\partial^{^\otimes}_x:(\A\otimes_{\HH M}\BB)_n(x)\to (\A\otimes_{\HH M}\BB)_{n-1}(x),$$ is defined on generators by the Leibniz formula
$$
\partial^{^\otimes}_{x}(a_z\otimes b_t)=\partial_za_z
\otimes b_t+
(-1)^{p}\,\,a_z\otimes \partial_tb_t.
$$
for all $z,t\in M$ such that $zt=x$,  $a_z\in \A_p(z)$,  $b_t\in \BB_q(t)$, and $p,q\geq 0$ such that $p+q=n$.

In this monoidal category, the unit object is $\Z$, defined in \eqref{z}, regarded as a complex concentrated in degree zero. The structure constraints
\begin{equation}\label{isocan2}\begin{array}{l}\boldsymbol{l}_\A:\Z\otimes_{\HH M} \A \cong  \A ,
\hspace{0.4cm}
\boldsymbol{c}_{\A,\BB}: \A \otimes_{\HH M} \BB \cong  \BB \otimes_{\HH M}  \A , \\
\boldsymbol{a}_{\A,\BB,\C}:  \A \otimes_{\HH M} ( \BB \otimes_{\HH M} \C)\cong
( \A \otimes_{\HH M} \BB )\otimes_{\HH M}\C,\end{array}
\end{equation}
 are respectively defined by the formulas
$$\begin{array}{l}\boldsymbol{l}_{xy}(x\otimes a_y)=x_*a_y,\\
\boldsymbol{c}_{xy}\big(a_x\otimes
b_y\big)=(-1)^{pq}\ b_y\otimes a_x,\\
\boldsymbol{a}_{xyz}\big( a_x\otimes (b_y\otimes
c_z
)\big)=(a_x\otimes b_y)\otimes c_z,\end{array}$$
for any $x,y,z\in M$, $a_x\in \A_p(x)$, $b_y\in \BB_q(y)$, and $c_z\in \C_r(z)$. As for $\HH M$-modules, we will treat these constraints as identities.

\subsection{Commutative differential graded algebras over $\HH M$}\label{dgh-mod}

A  commutative {\em differential graded algebra} (DG-algebra) $\A$ over $\HH M$ is defined to be a commutative monoid in the symmetric monoidal category of complexes of $\HH M$-modules, see \cite[Chapter VII, \S3]{Mac3}. Hence, $\A$ is a
complex of $\HH M$-modules equipped with a  {\em multiplication morphism} of
complexes  $\circ:\A\otimes_{\HH M}\A\to \A$ satisfying the
associativity $\circ(\circ\otimes id)=\circ(id\otimes \circ)$ and
the commutativity $\circ\,\boldsymbol{c}=\circ$,  and a {\em unit
morphism} of complexes $\iota:\Z \to \A$ satisfying
$\circ(\iota\otimes id_\A)=\boldsymbol{l}_\A$. We  write $$1=\iota_e(e)\in
\A_{0}(e)$$ and, for any $x,y\in M$, $a_x\in \A_{p}(x)$, and
$a_y\in \A_{q}(y)$,
$$a_x\circ a_y=\circ_{xy}(a_x\otimes a_y)\in \A_{p+q}(xy),$$
so that the algebra structure on the complex $\A$ gives us
multiplication  homomorphisms of abelian groups
$$\xymatrix{\A_{p}(x)\otimes \A_{q}(y)\to \A_{p+q}(xy), \hspace{0.5cm}
a_x\otimes a_y\mapsto a_x\circ a_y,}$$ and a {\em
unit} $1\in \A_{0}(e)$, satisfying
\begin{align}\label{dga1}
 x_*a_y\circ a_{z}&= x_*(a_y\circ a_z)=
a_y\circ x_*a_z,\\[3pt] \label{dga2}
a_x\circ a_y&= (-1)^{pq}\,\,a_y\circ a_x,\\[3pt]\label{dga3}
1\circ a_x&=a_x=a_x\circ 1,
\\[3pt]\label{dga4}
a_x\circ (a_y\circ a_z)&= (a_x\circ a_y)
\circ
a_z,\\[3pt]
\label{dga5}
 \partial_{xy}(a_x\circ a_y)&=\partial_xa_x\circ a_y+ (-1)^{p}\,\,
a_x\circ \partial_ya_y,
\end{align}
for all $x,y,z\in M$, $a_x\in \A_p(x)$, $a_y\in \A_q(y)$, and $a_z\in \A_r(z)$.

In these terms, a morphism $f:\A\to \BB$ of commutative DG-algebras over $\HH M$  is a morphism of complexes of $\HH M$-modules such that $f_{xy}(a_x\circ a_y)
=f_xa_x\circ f_ya_y$, and $f_e(1)=1$.

The category of commutative DG-algebras over $\HH M$ is symmetric
monoidal. The
tensor product of two of them $ \A \otimes_{\HH M} \BB $ is given by their tensor product
as complexes of $\HH M$-modules endowed with multiplication such that, for
$u,v,x,y\in M$, $a_u\in  \A_p(u)$, etc.,
$$
(a_u\otimes b_x)\circ (a_y\otimes b_z)=(a_u\circ a_y)\otimes (b_x\circ b_z)
$$
and with unit $1\otimes 1\in ( \A \otimes_{\HH M} \BB )_0(e)$.
Observe that  the canonical isomorphisms in \eqref{isocan2} are actually
of DG-algebras whenever the data  $ \A $, $ \BB $ and
$\C$ therein are DG-algebras over $\HH M$.

Commutative DG-algebras over $\HH M$ which are concentrated in degree zero are the same as  commutative monoids in the symmetric monoidal category of $\HH M$-modules, and they are simply called {\em algebras over $\HH M$} or {\em $\HH M$-algebras}. For example,  $\Z$  is an
$\HH M$-algebra  with multiplication the unit constraint
$\boldsymbol{l}:\Z\otimes_{\HH M}\Z\cong \Z$  and unit the identity $id:\Z\to\Z$. In other words, $\Z$ is an
$\HH M$-algebra whose unit is $e\in \Z e$  and whose
multiplication homomorphisms $\Z(x)\otimes \Z( y)\to \Z(xy)$
are given by $mx\circ ny=(mn)xy$, where $mn$ is multiplication of $m$ and $n$ in the ring $\Z$.

The augmented case is relevant. A commutative {\em differential
graded augmented algebra} (DGA-algebra) $\A$ over $\HH M$
is a commutative DG-algebra over $\HH M$ as above equipped with a homomorphism of commutative DG-algebras  (the {\em  augmentation}) $\epsilon:\A\to \Z$. Such an augmentation is entirely
determined by its component of degree 0, which is a morphism of
$\HH M$-algebras $\epsilon:\A_0\to \Z$ such that $\epsilon
\,\partial=0$.
Morphisms of commutative DGA-algebras over $\HH M$ are those of commutative  DG-algebras
which are compatible with the augmentations (i.e., $\epsilon f=\epsilon$).

\begin{remark}\label{eqabhhm2}{\em  When  $M=G$ is a group, the equivalence between the category of abelian groups and the category of $\HH G$-modules, described in Remark \ref{eqabhhm}, is symmetric monoidal and, therefore, produces an equivalence between the category of commutative DGA-rings and the category of commutative DGA-algebras over $\HH G$. Thus every commutative DGA-ring $A$ defines a {\em constant} commutative DGA-algebra over $\HH G$, equally denoted by $A$, and each commutative DGA-algebra over  $\HH G$, $\A$, gives rise to the DGA-ring $\A(e)$, which comes with a natural isomorphism of DGA-algebras $\A\cong \A(e)$ whose component at each $x\in G$ is the isomorphism of augmented chain complexes $x_*^{-1}:\A(x)\to \A(e)$.
}
\end{remark}

\section{The Bar construction on  commutative  DGA-algebras over $\HH M$}\label{sect3}

Let $\A$ be any given commutative DGA-algebra over $\HH M$. As we explain below, $\A$ determines a new commutative DGA-algebra over $\HH M$, denoted by $\B(\A)$ and called the {\em bar construction} on $\A$.

Previously to describe $\B(\A)$,  let us introduce  complexes of $\HH M$-modules $\bar{\A}$, $\s\bar{\A}$, and $\T^p\s\bar{\A}$ for each integer $p\geq 0$, and a double complex  of $\HH M$-modules $\T^\bullet\s\A$, as follows:

The {\em reduced complex} $\bar{\A}=\cdots \to \bar{\A}_2\overset{\partial}\to \bar{\A}_1\overset{\partial}\to \bar{\A}_0$
is defined to be the cokernel of the unit morphism $\iota:\Z\to \A$. That is,
$
\bar{\A}=\cdots \to \A_2\overset{\partial}\to \A_1\overset{\partial}\to \A_0/\iota \Z
$.
Note that $\iota$
embeds  $\Z$ as a direct summand of the underlying complex $\A$, since, being $\epsilon:\A\to \Z$ the augmentation, $\epsilon\iota=id_{\Z}$. We will use below the following notation: For any $x\in M$ and each chain $a_x$ of the chain complex $\A(x)$, $\tilde{\epsilon}(a_x)$ is the integer which express $\epsilon_x(a_x)$ as a multiple of the generator $x$ of the abelian group $\Z(x)$, that is, such that
\begin{equation}\label{notep}
\epsilon_x(a_x)=\tilde{\epsilon}(a_x)x.
\end{equation}

The complex $\s\bar{\A}$ is the {\em suspension} of $\bar{\A}$, that is, the
complex of $\HH M$-modules defined by $(\s\bar{\A})_{p+2}=\A_{p+1}$, $(\s\bar{\A})_{1}=\A_0/\iota\Z$, $(\s\bar{\A})_0=0$, and differential $-\partial$. The {\em suspension map} is then the morphism of complexes $\s:\bar{\A}\to \s\bar{\A}$, of degree 1, defined by $$\s_p=id_{\bar{\A}_p}:\bar{\A}_p\to (\s\bar{\A})_{p+1}=\bar{\A}_p.$$ Note that the sign in the differential of $\s\bar{\A}$ is taken so that the equality $\partial \s =-\s \partial$ holds.

For each $p\geq 1$, let $\T^p\s\bar{\A}$ be the complex of $\HH M$-modules defined by the iterated tensor product
$$
\xymatrix{\T^p\s\bar{\A}=\s\bar{\A}\otimes{}_{\!\HH M}
\cdots  \otimes{}_{\!\HH M}\s\bar{\A}& (p \text{ factors}).}
$$
Thus, for any integer $n\geq 0$ and $x\in M$, the abelian group $(\T^p\s\bar{\A})_n(x)$ is generated by elements $\s\bar{a}_{x_1}\otimes \cdots\otimes \s\bar{a}_{x_p}$, that we write as
\begin{equation}\label{barnot}
[a_{x_1}\!\mid \cdots\mid \!a_{x_p}],
\end{equation}
where the $x_i\in M$ are elements of the monoid such that $x_1\cdots x_p=x$, and the $a_{x_i}\in \A_{r_i}(x_i)$ are chains of the complexes of abelian groups $\A(x_i)$
whose degrees satisfy that $p+r_1+\cdots +r_p=n$.  On such a generator \eqref{barnot}, the differential $\partial^{^{\otimes}}$ of $\T^p\s\bar{\A}$ at $x$,  $$\partial^{^\otimes}_x:(\T^p\s\bar{\A})_n(x)\to (\T^p\s\bar{\A})_{n-1}(x),$$ acts by
\begin{equation*}
\partial^{^{\otimes}}_x[a_{x_1}|\cdots|a_{x_p}]=
 -\sum_{i=1}^{p}(-1)^{e_{i-1}}[a_{x_1}|\cdots|a_{x_{i-1}}|\partial_{x_i}a_{x_i}
|a_{x_{i+1}}|\cdots|a_{x_p}],
\end{equation*}
where the exponents $e_i$ of the signs are $e_0=0$ and, for $i\geq 1$,
$$ e_i=i+r_1+\cdots+r_i,
$$
and $\partial_{x_i}:\A_{r_i}(x_i)\to \A_{r_i-1}(x_i)$ is the differential of $\A$ at $x_i$.
Remark that the elements \eqref{barnot}  are normalized, in the sense
that
$
[a_{x_1}\!\mid \cdots\mid \!a_{x_p}]=0
$
whenever some $a_{x_i}=x_{i*}1\in \A_0(x_i)$.

For $p=0$, we take $\T^0\s\bar{\A}$ to be $\Z$, but where we write $[\ \, ]$ for the unit $e\in \Z(e)$.
Thus, $\T^0\s\bar{\A}$ is the concentrated in degree 0 complex of $\HH M$-modules such that, for any $x\in M$,  $\T^0\s\bar{\A}(x)$ is
the free abelian group on the element $x_*[\ \, ]$ ($=[\, \ ]$ if $x=e$), and, for each $x,y\in M$, $y_*:\T^0\s\bar{\A}(x)\to \T^0\s\bar{\A}(xy)$ is determined by $y_*x_*[\ \, ]=(yx)_*[\ \, ]$.

The double complex  of $\HH M$-modules
$$
\T^\bullet\s\A= \cdots \to\T^2\s\bar{\A}\overset{\partial^\circ}\longrightarrow \T^1\s\bar{\A}\overset{\partial^\circ}\longrightarrow \T^0\s\bar{\A}
$$
is then constructed, thanks to the multiplication $\circ$ in $\A$, by the morphisms of complexes of $\HH M$-modules $\partial^\circ:\T^p\s\bar{\A}\to\T^{p-1}\s\bar{\A} $, which are of degree $-1$ (so that $\partial^\circ\partial^{^\otimes}=-\partial^{^\otimes}\partial^\circ$) and defined, at any $x\in M$, by the homomorphisms
$$
\partial^\circ_x: (\T^p\s\bar{\A})_n(x)\to (\T^{p-1}\s\bar{\A})_{n-1}(x)
$$
given on generators as in \eqref{barnot} by
\begin{align*}
\partial^\circ_x[a_{x_1}|\cdots|a_{x_p}]=&\
\tilde{\epsilon}_{x_1}\!(a_{x_1})\ x_{1*}[a_{x_2}|\cdots|a_{x_p}]\\
\nonumber
&+\sum_{i=1}^{p-1}(-1)^{e_i}[a_{x_1}|\cdots|a_{x_{i-1}}|a_{x_i}
\circ a_{x_{i+1}}|a_{x_{i-1}}|\cdots|a_{x_p}]\\[5pt] \nonumber
& +(-1)^{e_p}\  \tilde{\epsilon}_{x_p}\!(a_{x_p})\, x_{p\,*}[a_{x_1}|\cdots|
a_{x_{p-1}}]
\end{align*}
(recall the notation $\tilde{\epsilon}$ from \eqref{notep}, and note that the first (resp. last) summand in the above formula is zero whenever the degree $r_1$  of $a_{x_1}$ in the chain complex
$\A(x_1)$ (resp. ($r_p$ of $a_{x_p}$)) is higher than zero).

All in all, we are now ready to present the bar construction $\B(\A)$. As a graded $\HH M$-module
$$
\B(\A)= \cdots \to\B(\A)_2\overset{\partial}\longrightarrow \B(\A)_1\overset{\partial}\longrightarrow \B(\A)_0
$$
is defined by the $\HH M$-modules
$$
\xymatrix{\B(\A)_n=\bigoplus_{p\geq 0}(\T^p\s\bar{\A})_n}.
$$
Notice that $\partial^{^\otimes}\B(\A)_n\subseteq \B(\A)_{n-1}$,  $\partial^{\circ}\B(\A)_n\subseteq \B(\A)_{n-1}$, and that $(\partial^{^\otimes}+\partial^\circ)^2=0$. Thus, $\B(\A)$ becomes a complex of $\HH M$-modules with differential
$$
\partial=\partial^{^\otimes}+\partial^\circ:\B(\A)_{n}\to \B(\A)_{n-1}.
$$

\begin{proposition}  $\B(\A)$ is a commutative DGA-algebra over $\HH M$,
with  multiplication $$\circ:\B(\A)\otimes_{_{\HH\!M}}
\B(\A) \to \B(\A)$$
defined, for integers $m,n\geq 0$ and $x,y\in M$, by the
homomorphisms of abelian groups
$$
\xymatrix{\circ:\B(\A)_{m}(x)\otimes \B(\A)_{n}(y)\to
\B(\A)_{m+n}(xy)}
$$
given by the {\em shuffle products}
$$
[a_{x_1}|\cdots|a_{x_p}]\circ
[a_{x_{p+1}}|\cdots|a_{x_{p+q}}]=\sum_\sigma (-1)^{e(\sigma)}
\big[a_{x_{\sigma^{\scriptsize {\text -1}}(1)}}|\cdots|
a_{x_{\sigma^{ {\text -1}}(p+q)}}\big]
$$
for any $x_i\in M$ and $a_{x_i}\in \A_{r_i}(x_i)$, $i=1,\ldots,p+q$,  such that $x_1\cdots x_p=x$, $x_{p+1}\cdots x_{p+q}=y$, $p+\sum_{i=1}^{p} r_i=m$, and $q+\sum_{j=1}^{q} r_{p+j}=n$,
where the sum is taken over all $(p,q)$-shuffles $\sigma$ and, for each $\sigma$, the exponent of the sign is $e(\sigma)=\sum (1+r_i)(1+r_{p+j})$ summed over all pairs $(i,p+j)$
such that $\sigma(i)>\sigma(p+j)$.

The unit is $[\ \,]\in \B(\A)_0(e)$, that is, the unit morphism $\iota:\Z\to \B(\A)$ is
the isomorphism of $\HH M$-modules $\iota:\Z\cong \B(\A)_0$ given by $\iota_x(x)=x_*[\ \,]$, for any $x\in M$,  and the augmentation $\epsilon:\B(\A)\to \Z$ is defined by the isomorphism of $\HH M$-modules $\epsilon=\iota^{-1}:\B(\A)_0\cong \Z$ such that $\epsilon_x(x_*[\ \,])=x$, for any $x\in M$.
\end{proposition}
\begin{proof}
We give an indirect proof, by using that the category of $\HH M$-modules is closely related
to the category $\Z M\text{-}\mathrm{Mod}$, of ordinary modules over the monoid ring $\Z M$.

There is
an exact faithful functor $\Gamma:\HH M\text{-}\mathrm{Mod} \to
\Z M\text{-}\mathrm{Mod}$, which carries any
$\HH M$-module $\A$ to the $\Z M$-module
defined by the abelian group
$
\xymatrix{\Gamma \A=\bigoplus_{x\in M} \A(x),}
$
with $M$-action of an element $y\in M$ on an element $a_x\in \A(x)$
given by
$
y\,a_x=y_*a_x\in \A(xy)
$. This functor $\Gamma $ is left adjoint to the functor which
associates to any $\Z M$-module $A$ the constant on objects $\HH M$-module defined by the underlying abelian group $A$, with $y_*:A\to A$, for any $y\in M$, the homomorphism
 of multiplication by $y$ \cite{kur-Pir}.

It is plain to see that $\Gamma $ is a symmetric strict
monoidal functor, that is, $\Gamma \Z=\Z M$,  for any $\HH M$-modules
$\A$ and $\BB$, $\Gamma (\A\otimes_{\HH M}\BB)=\Gamma \A\otimes_{\Z M}\Gamma \BB
$, and it carries the associativity, unit, and commutativity constraints of the monoidal category of $\HH M$-modules to the corresponding ones of the category of $\Z M$-modules. Then, the same properties hold for the induced functor $\Gamma$ from the symmetric monoidal category of complexes of $\HH M$-modules to the the symmetric monoidal category of complexes of $\Z M$-modules. It follows that $\Gamma$ transform commutative monoids in the category of complexes of $\HH M$-modules (i.e. commutative DG-algebras over $\HH M$) to commutative  monoids in the category of $\Z M$-modules (i.e., commutative DG-algebras over $\Z M$), and therefore $\Gamma$ also transform commutative DGA-algebras over $\HH M$ to commutative  DGA-algebras over the monoid ring $\Z M$.

Now, given $\A$, a commutative  DGA-algebra over $\HH M$, let $\B(\Gamma \A)$ be the commutative DGA-algebra over $\Z M$ obtained by applying the ordinary Eilenberg-Mac Lane bar construction on $\Gamma \A$ \cite[Chapter X, Theorem 12.1]{maclane}. A direct comparison shows that $\B(\Gamma \A)=\Gamma\B(\A)$ as complexes of $\Z M$-modules, and also that  its multiplication, unit, and augmentation  are, respectively, just the morphisms $$\xymatrix{\B(\Gamma A)\otimes_{\Z M}\B(\Gamma \A)=\Gamma\big(\B(\A)
\otimes_{\HH M}\B(\A)\big)\ar[r]^-{\textstyle \Gamma \circ}& \Gamma\B(\A)=\B(\Gamma \A)},$$
$$\xymatrix{\Z M=\Gamma \Z\ar[r]^-{\Gamma \iota}& \Gamma\B(\A)=\B(\Gamma \A),&\B(\Gamma \A)=\Gamma\B(\A)
\ar[r]^-{\Gamma \epsilon}&\Gamma \Z=\Z M. }
$$
Then, as $\B(\Gamma A)$ is actually a commutative  DGA-algebra over $\Z M$, it follows that the equalities
$$
\Gamma\big(\circ(\circ\otimes id_{\B(\A)})\big)=\Gamma\big(\circ(id_{\B(\A)}\otimes \circ)\big), \hspace{0.4cm}\Gamma(\circ\,\boldsymbol{c_{_{\B(\A),\B(\A)}}})=\Gamma \circ,$$
$$\Gamma(\circ(\iota\otimes id_{\B(\A)}))=\Gamma\boldsymbol{l}_{\B(\A)}, \hspace{0.4cm} \Gamma(\circ (\epsilon\otimes \epsilon))=\Gamma(\epsilon \circ), \hspace{0.4cm} \Gamma(\epsilon \iota)=\Gamma id_\Z.
$$
hold. Therefore, the result, that is, that $\B(\A)$ is a commutative  DGA-algebra over $\HH M$, follows since  the functor $\Gamma$ is faithful.
\end{proof}

\begin{remark}{\em
Observe, as in \cite[\S 7]{E-M-I},  that the shuffle product $\circ$ on $\B(\A)$ can also be expressed by the recursive formula below, where $\alpha=[a_{x_1}|\cdots|a_{x_p}]\in \B(\A)_{r}(x)$, $\beta=[b_{y_1}|\cdots|b_{y_q}]\in \B(\A)_{s}(y)$,  $a_{z}\in \A_{m}(z)$ and  $b_{t}\in \A_{n}(t)$.
\begin{equation}\label{forshufpro}
[\alpha\mid a_{z}]\circ [\beta\mid b_{t}]=[[\alpha\mid a_{z}]\circ \beta\mid b_{t}]
+(-1)^{r(n+s+1)}[\alpha\circ [\beta\mid b_{t}]\mid a_{z}]
\end{equation}
}

\end{remark}

Let us stress  the {\em
suspension} morphism of complexes of $\HH M$-modules, of degree $1$ (hence satisfying $\partial\, \s =-\s\, \partial$),
\begin{equation}\label{sus}
\s: \A\to \B(\A),
\end{equation}
which is defined, at any $x\in M$, by
$\s_xa_x=[a_x]\in \B(\A)(x)$, for any chain $a_x$ of $\A(x)$.

\vspace{0.2cm}
Such as Mac Lane did in \cite[Chapter X, \S 12]{maclane} for ordinary commutative DGA-algebras over a commutative ring, the cohomology of a commutative DGA-algebra over $\HH M$
can be defined in ``stages" or ``levels". If $\A$ is any commutative DGA-algebra over $\HH M$, then $\B(\A)$ is again a commutative DGA-algebra over $\HH M$,
so an iteration is possible to form $\B^{r}\!(\A)$ for each integer $r\geq 1$. Hence,  we define the {\em $r$th level cohomology groups of $\A$ with coefficients in a $\HH M$-module $\BB$}, denoted by $H^n(\A,r;\BB)$, as
\begin{equation*}
H^n(\A,r;\BB)=H^n\big(\mathrm{Hom}_{\HH M}(\B^{r}\!(\A),\BB)\big), \hspace{0.4cm} n=0,1,\ldots \, ,
\end{equation*}
where $\mathrm{Hom}_{\HH M}(\B^{r}\!(\A),\BB)$ is the cochain complex obtained by applying the functor $\mathrm{Hom}_{\HH M}(-,\BB)$ to the underlying chain complex of $\HH M$-modules  $\B^{r}\!(\A)$.

\begin{remark}\label{eqabhhm3} {\em
When the bar construction above is applied on the constant DGA-algebra over $\HH M$ defined by a commutative DGA-ring $A$, the result is just the
constant DGA-algebra over $\HH M$ defined by the commutative DGA-ring
obtained by applying on $A$ the Eilenberg-Mac Lane
reduced bar construction. Hence, the notation $\B(A)$ is not
confusing.}
\end{remark}

If $\A$ and $\BB$ are commutative DGA-algebras over $\HH M$, then we say that two morphisms of DGA-algebras $f:\A\to \BB$ and $g:\BB\to \A$ form a {\em contraction} whenever $fg=id_\BB$, and there exists a homotopy of morphisms of complexes $\Phi:gf\Rightarrow id_\A$ satisfying the conditions
\begin{equation}\label{contr2}
  \Phi g=0, \ f \Phi =0,\  \Phi\Phi=0.
 \end{equation}

Paralleling the proof by Eilenberg and Maclane of \cite[Theorem 12.1]{E-M-I}, one proves the following:

\begin{lemma}\label{iterate}
 If $f:\A\to \BB$ and $g:\BB\to \A$ form a contraction of commutative  DGA-algebras over $\HH M$, then the induced
 $\B(f):\B(\A)\to \B(\BB)$ and $\B(g):\B(\BB)\to \B(\A)$ also form a contraction.
\end{lemma}

\section{Free $\HH M$-modules}\label{sect4}
 Let $\mathbf{Set}\!\!\downarrow\!\!_M$ be
the comma category of sets over
 the underlying set of $M$; that is, the category whose objects
 $S=(S,\pi)$ are sets $S$ endowed with a map $\pi:S\to M$, and
 whose morphisms are maps $\varphi:S\to T$ such that $\pi\varphi=\pi$. There is a {\em forgetful functor}
 $$\U:\HH M\text{-Mod}
 \to \mathbf{Set}\!\!\downarrow\!\!_M,$$
 which carries any $\HH\!M$-module $\A$ to the disjoint union set
 $$\xymatrix{\U \A=\bigcup\limits_{x\in M}\A(x)=\{(x,a_x)\mid \,  x\in M,\, a_x\in \A(x)\},}$$
endowed with the projection map $\pi:\U \A\to M$, $\pi(x,a_x)=x$. A morphism  $f:\A\to \BB$ is sent to the map $\U f:\U \A\to \U \BB$ given by $\U f(x,a_x)=(x,f_xa_x)$.
 There is also   a {\em free  $\HH\!M$-module}
functor
\begin{equation}\label{zz}\ZZ :\mathbf{Set}\!\!\downarrow\!\!_M \to
\HH\!M\text{-Mod}, \end{equation}
 which is defined as follows:
If $S$ is any set over $M$, then $\ZZ S$ is the $\HH M$-module such that, for each  $x\in M$,
$$ \ZZ S(x)=\Z\{(u,s)\in M\times S \mid u\,\pi(s)=x\}
$$
is the free abelian group with generators all pairs $(u,s)$, where
$u\in M$ and $s\in S$,  such that $u\,\pi(s)=x$.
We usually write $(e,s)$ simply by $s$; so that each
element of $s\in S$ is regarded as an element $s\in  \ZZ S(\pi s)$.
For any
$x,y\in M$,  the homomorphism
$$ y_*: \ZZ S(x)\to  \ZZ S(xy)$$
is defined on generators by $y_*(u,s)=(uy,s)$.
If $\varphi:S\to T$ is any map of sets over $M$, the induced
morphism  $ \ZZ \varphi: \ZZ S\to  \ZZ T$ is given, at
each $x\in M$, by the homomorphism $ (\ZZ \varphi)_x: \ZZ S(x)\to  \ZZ T(x)$ defined on generators by $( \ZZ\varphi)_x(u,s)=(u,\varphi s)$.

\begin{proposition}\label{adfu} The functor $\ZZ$  is left adjoint to the functor
 $\U$. Thus, for $S$ any set over $M$, to
each $\HH M$-module $\A$ and each list of elements $a_s\in
\A(\pi s)$, one for each $s\in S$, there is a unique morphism of
$\HH M$-modules $f:\ZZ S\to \A$ with
 $f_{\pi s}(s)=a_s$ for every $s\in S$.
\end{proposition}
\begin{proof} At any set $S$ over $M$, the unit of the
adjunction is the map
$$\nu:S\to \U\ZZ S=
\{(x,a_x)\mid x\in M,\, a_x\in  \ZZ S(x)\},
\hspace{0.5cm} s\mapsto (\pi s, s).
$$
If $\A$ is a $\HH M$-module and $\varphi:S\to \U \A$ is any map over
$M$, then, the unique morphism of $\HH M$-modules $f: \ZZ S\to \A$
such that $( \U f)\,\nu =\varphi$ is determined by the equations
$f_x(u,s)=u_*\varphi(s)$, for any $x\in M$ and  $(u,s)\in M\times S$ with $u\,\pi(s)=x$.
\end{proof}

 The category $\mathbf{Set}\!\downarrow\!_M$  has a symmetric monoidal structure, where the tensor product of two sets over $M$, say $S$ and $T$, is the cartesian
product set of $S\times T$  with $\pi(s,t)=\pi(s)\pi(t)$. The unit
object is provided by the unitary set $\{e\}$ with $\pi(e)=e\in M$, and
the associativity, unit,  and commutativity constraints
are the obvious ones. Hereafter, the category
$\mathbf{Set}\!\downarrow\!\!_M$ will be considered with this monoidal
structure\footnote{The category $\mathbf{Set}\!\downarrow\!_M$ has a different monoidal structure where the tensor product  is given by the fibre-product
$S\times_MT$ with $\pi(s,t)=\pi(s)=\pi(t)$.}.

\begin{proposition}\label{fsmon}
The free $\HH M$-module functor \eqref{zz} is symmetric monoidal, that is,
there are natural and coherent isomorphisms of $\HH M$-modules
\begin{equation*}\xymatrix{
\ZZ(S\times T)\cong \ZZ S\otimes_{\HH M}\ZZ T,
\hspace{0.4cm} \ZZ\{e\}\cong \Z ,}
\end{equation*}
for $S$ and $T$ any sets over $M$.
\end{proposition}
\begin{proof} For $S$, $T$ any given sets over $M$,  the isomorphism $f:\ZZ(S\times
T)\cong \ZZ S\otimes_{\HH M}\ZZ T$ is the morphism of $\HH M$-modules
such that, for any $(s,t)\in S\times T$,  $f_{\pi(s,t)}(s,t)=s\otimes t$. Observe  that, for any $x\in M$,  the abelian group $\ZZ(S\times
T)(x)$ is free with generators the elements $(u,s,t)=u_*(s,t)$, with $u\in M$, $s\in S$, and $t\in T$, such that  $u\,\pi(s)\,\pi(t)=x$, while
$(\ZZ S\otimes_{\HH M}\ZZ T)(x)$ is the abelian group generated by the
elements $(u,s)\otimes (v,t)=u_*s\otimes v_*t$, with $u,v\in M$, $s\in S$, and
$t\in T$, such  that $u\,\pi(s)\, v\, \pi(t)=x$, with the relations
$u_*s\otimes v_*t=(uv)_*(s\otimes t)$. Then, the
homomorphism $f_x:\ZZ(S\times T)(x)\to (\ZZ S\otimes_{\HH\! M}\ZZ T)(x)$,
which acts on elements of the basis  by
$f_x(u_*(s,t))=u_*(s\otimes t)$,
is clearly an isomorphism of abelian groups.

The isomorphism $f:\ZZ\{e\}\cong \Z $ is the
morphism of $\HH M$-modules such that  $f_e(e)=e$. Observe that, for any $x\in M$, the isomorphism $f_x$ is the composite
$$
\ZZ\{e\}(x)=\Z\{(u,e)\mid ue=x\}=\ZZ\{(x,e)\}\cong \Z\{x\}=\Z(x).
$$

It is straightforward to see that the isomorphisms $f$ above are
natural and coherent, so that $\ZZ$ is actually a symmetric monoidal
functor.
\end{proof}

\begin{corollary} For $S$ and $T$ any two sets over $M$, the tensor product $\HH M$-module $\ZZ S\otimes_{\HH M}\ZZ T$ is free  on the set of elements
$s\otimes t$,  $s\in S$, $t\in T$, with  $\pi(s\otimes t)=\pi(s)\pi(t)$.
\end{corollary}

Since the functor $\ZZ$ is symmetric
monoidal, it transports commutative monoids in
$\mathbf{Set}\!\downarrow\!_M$ to commutative monoids in
$\HH M\text{-}\mathrm{Mod}$, that is, to algebras over $\HH M$.
As a commutative monoid in the symmetric monoidal category $\mathbf{Set}\!\downarrow\!_M$ is merely a commutative monoid over $M$, that is,  a commutative monoid $S$ endowed with a homomorphism $\pi:S\to M$, the corollary below follows.

 \begin{corollary}\label{zstruc} If $S$ is a commutative monoid over $M$, then  the free $\HH M$-module $\ZZ S$ is an algebra over $\HH M$.  The multiplication morphism $\circ:\ZZ S\otimes_{\HH M}\ZZ S\to \ZZ S$ is the  composite
 $$\xymatrix{\ZZ S\otimes_{\HH M}\ZZ S\cong \ZZ(S\times S)\ar[r]^-{\ZZ \mathrm{m}}& \ZZ S,}$$
where $\mathrm{m}:S\times S\to S$ is the homomorphism of multiplication in $S$, $\mathrm{m}(s,s')=ss'$,
and the unit morphism $\iota:\ZZ \to \ZZ S$ is the  composite
 $\xymatrix{\Z \cong \ZZ\{e\}\ar[r]^-{\ZZ \mathrm{i}}& \ZZ S}$,
where $\mathrm{i}:\{e\}\to S$ is the trivial homomorphism mapping the unit  of $M$ to the unit  of $S$.
\end{corollary}

\section{The cohomology groups $H^n(M,r;\A)$}\label{sect5}
Let us consider the commutative monoid $M$ over itself with $\pi=id_M:M\to M$. Then, by Corollary \ref{zstruc},  the free $\HH M$-module $\ZZ M$ is an algebra over $\HH M$. Explicitly, this is described as follows: For each $x\in M$,
$$\ZZ M(x)=\Z\{ (u,v) \mid uv=x\}$$ is the free
abelian group with generators all pairs $(u,v)\in M\times M$
such that $uv=x$.  For any $x, y\in M$, the homomorphism $y_*:\ZZ M(x)\to \ZZ M(xy)$ is given on generators by
 $y_*(u,v)=(yu,v)$, and the homomorphism of multiplication $$\circ: \ZZ M(x)\otimes \ZZ M(y)\to \ZZ M(xy)$$ is defined on generators by
$\xymatrix{(u,v)\otimes (w,t)\mapsto (u,v)\circ (w,t)= (uw,vt),}$  for any $u,v,w,t\in M$ such that $uv=x$ and $wt=y$. The unit is  $(e,e)\in \ZZ M(e)$.
We see each element $x\in M$ as an element of $\ZZ M(x)$ by means of the identification $x=(e,x)$, so that that any generator $(u,v)$ of $\ZZ M(x)$ can be write as $u_*v$.

By Proposition \ref{adfu}, if $\A$ is any $\HH M$-module, for any list of elements $a_x\in \A(x)$, one for each $x\in M$, there is an unique morphism of $\HH M$-modules $f:\ZZ M\to \A$ such that each homomorphism $f_x:\ZZ M(x)\to \A(x)$ verifies that $f_x(x)=a_x$ (explicitly, $f_x$ acts on generators by $f_x(u,v)=u_*a_v$). Furthermore, it is plain to see that, if $\A$ is an algebra over $\HH M$, then $f$ is a morphism of algebras if and only if $a_e=1$ and $a_x\circ a_y=a_{xy}$ for all $x,y\in M$.

Hereafter, we regard $\ZZ M$ as a commutative  DGA-algebra over $\HH M$ with the
trivial grading, that is, with $(\ZZ M)_n=0$ for $n>0$ and
$(\ZZ M)_0=\ZZ M$, and with augmentation the morphism of
$\HH M$-algebras $$\epsilon :\ZZ M\to \Z ,$$
such that, for any $x\in M$, $\epsilon_x(x)=x\in \Z(x)$. Then,  we define, for each integer $r\geq 1$,
{\em the $r$th level cohomology groups of the commutative monoid $M$ with coefficients in a $\HH M$-module $\A$} by
\begin{equation}\label{defhmra}
H^n(M,r;\A)=H^n(\ZZ M,r;\A), \hspace{0.4cm} n=0,1,\ldots \, ,
\end{equation}
or, in other words,
$$
H^n(M,r;\A)=H^n\big(\mathrm{Hom}_{\HH\! M}(\B^{r}\!(\ZZ M),\A)\big),
$$
where $\mathrm{Hom}_{\HH\! M}(\B^r\!(\ZZ M),\A)$ is the cochain complex obtained by applying the abelian group valued functor $\mathrm{Hom}_{\HH\! M}(-,\A)$ to the neglected chain complex of $\HH M$-modules  $\B^r\!(\ZZ M)$.

\begin{remark}\label{eqabhhm4} {\em When $M=G$ is an abelian group, $\ZZ G$ is isomorphic to the constant DGA-algebra  over $\HH G$ defined by the commutative DGA-ring $\ZZ G(e)$ (see Remark \ref{eqabhhm2}), which is itself isomorphic to the trivially graded  DGA-ring defined by the group ring
$\Z G$ with augmentation the ring homomorphism $\alpha:\Z G\to \Z$ such that $\alpha(x)=1$ for any $x\in G$. To see this, observe that $\ZZ G(e)$ is the commutative ring whose underlying abelian
group is freely generated by the elements of the form $(x^{-1},x)$, $x\in G$, with multiplication such that $(x^{-1},x)\circ (y^{-1},y)=((xy)^{-1},xy)$, and
unit $(e,e)=e$. The map $(x^{-1},x)\mapsto x$ clearly
determines a ring isomorphism between $\ZZ G(e)$ and the group ring
$\Z G$, which is compatible with the corresponding augmentations.

Hence, for any integer $r\geq 1$, $\B^r\!(\ZZ G)\cong \B^r\!(\Z G)$  (see Remark \ref{eqabhhm3})\footnote{The commutative DGA-rings $ \B^r\!(\Z G)$ are denoted by  $ A_{_N}(G,r)$ in \cite{E-M-I}}, and therefore for any abelian group $A$, regarded as a constant $\HH G$-module,
there are natural isomorphisms $$\mathrm{Hom}_{\HH G}(\B^r\!(\ZZ G),A)\cong \mathrm{Hom}_{\HH G}(\B^r\!(\Z G),A)\cong \mathrm{Hom}(\B^r\!(\Z G),A)$$
showing that the $r$th level cohomology groups $H^n(G,r;A)$ in \eqref{defhmra} agree with those by Eilenberg and Mac Lane in \cite{E-M-I},
which compute the cohomology of the spaces $K(G,r)$ by means of natural isomorphisms $H^n(K(G,r),A)\cong H^n(G,r;A)$.
}
\end{remark}

From here on, this section is  dedicated to show explicit cochain descriptions for some of these cohomology groups, starting with those of first level
  $$H^n(M,1;\A)=H^n\big(\mathrm{Hom}_{\HH\! M}(\B(\ZZ M),\A)\big).$$
Let us analyze the underlying complex $\B(\ZZ M)$.
For any integer $n\geq 1$,
$$\B(\ZZ M)_n=\overline{\ZZ M}\xymatrix{\otimes_{\HH\!M}}
\overset{(n\text{ factors})}\cdots
\xymatrix{\otimes_{\HH \!M}}\overline{\ZZ M},$$ where
$\overline{\ZZ M}=\ZZ M/\iota\Z=\ZZ M/\ZZ \{e\}\cong \ZZ M^*$ is a free $\HH M$-module on
$M^*=M\setminus\{e\}$ with $\pi:M^*\to M$ the
inclusion map. Then, by construction and Proposition \ref{fsmon}, we have
that

\vspace{0.2cm}
\begin{quote} $\bullet$ {\em The $\HH M$-module $\B(\ZZ M)_0$ is free on the unitary set $\{[\ \,]\}$ with $\pi[\ \,]=e$ and,  for any $n\geq 1$,  $\B(\ZZ M)_n$ is a
free $\HH M$-module generated by the set over $M$ consisting of
$n$-tuples of elements of $M$
$$
\alpha_n=[x_1|\cdots|x_n], \hspace{0.4cm} \text{with }\ \pi\alpha_n=x_1\cdots
x_n,
$$
 which we call {\em generic $n$-cells} of $\B(\ZZ M)$, with the relations
$\alpha_n=0$ whenever some $x_i=e$.}
\end{quote}
\begin{quote} $\bullet$ {\em The differential
$\partial:\B(\ZZ M)_n\to \B(\ZZ M)_{n-1}$ is the morphism of
$\HH M$-modules such that, for each $x\in M$ and any generic $n$-cell $[x_1|\cdots|x_n]$ with $x_1\cdots x_n=x$,
\begin{align}\nonumber
\partial_x[x_1|\cdots|x_n]=&\
 x_{1*}[x_2|\cdots|x_n]
+\sum_{i=1}^{n-1}(-1)^{i}[x_1|\cdots|x_i x_{i+1}|\cdots|x_n]\\[4pt] \nonumber
& +(-1)^{n} x_{n*}[x_1|\cdots|
x_{n-1}].
\end{align}
}
\end{quote}

Hence, Proposition \ref{adfu} gives the following.

\begin{theorem}For any $\HH M$-module $\A$, the cohomology groups $H^n(M,1;\A)$
can be computed as the cohomology groups of the {\em cochain complex of normalized $1$st level cochains of $M$ with values in $\A$},
\begin{equation}\label{cm1a}C(M,1;\A):\
0\to C^0(M,1;\A)\overset{\partial^0}\longrightarrow C^1(M,1;\A)\overset{\partial^1}\longrightarrow C^2(M,1;\A)\overset{\partial^2}\longrightarrow \cdots,
\end{equation}
where

$\bullet$  $C^0(M,1;\A)=\A(e)$, and for $n\geq 1$, $C^n(M,1;\A)$ is the abelian group, under pointwise addition,
of functions $$\xymatrix{f:M^n\to \bigcup_{x\in M}\A(x)}$$ such that $f(x_1,\ldots,x_n)\in \A(x_1\cdots x_n)$ and $f(x_1,\ldots,x_n)=0$ whenever some $x_i=e$,

 $\bullet$  $\partial^0=0$, and for $n\geq 1$, the coboundary $\partial^n:C^n(M,1;\A)\to C^{n+1}(M,1;\A)$ is given by
\begin{align} \nonumber
(\partial^{n} f)(x_1,\cdots,x_{n+1})=&\
 x_{1*}f(x_2,\cdots,x_{n+1})
+\sum_{i=1}^{n}(-1)^{i}f(x_1,\cdots,x_i x_{i+1},\cdots,x_{n+1})\\[4pt] \nonumber
& +(-1)^{n+1} x_{n+1*}f(x_1,\cdots,x_{n}).
\end{align}

\end{theorem}

Let us now recall that the so-called Leech cohomology groups \cite{leech} of a (not necessarily commutative) monoid $M$,
which we denote here by $ H^n_{^\mathrm{L}}(M,\A)$, take coefficients in $\DD M$-modules, that is, in abelian group valued functors on the category $\DD M$,
whose set of objects is $M$ and set of arrows $M\times M\times M$, where $(x,y,z):y\to xyz$. Composition is given by $(u,xyz,v)(x,y,z)=(ux,y,zv)$,
and the identity morphism of any object $x$ is $(e,x,e):x\to x$.
For any $\DD M$-module $\A:\DD M\to \mathbf{Ab}$, if  we write $\A(x,y,z)=x_*z^*:\A(y)\to \A(xyz)$, then we see that $\A$ consists of abelian groups $\A(x)$, one for each $x\in M$, and homomorphisms
$$ x_*:\A(y)\to \A(xy), \hspace{0.2cm} x^*:\A(y)\to \A(yx),$$
 for each $x,y\in M$, such that the equations below hold.
$$\begin{array}{cc}x_*y_*=(xy)_*: \A(z)\to \A(xyz),& y^*x^*=(xy)^*: \A(z)\to \A(zxy),\\[3pt]
e_*=e^*=id_{\A(x)}:\A(x)\to \A(x),&
x_*y^*=y^*x_*:\A(z)\to \A(xzy).
\end{array}$$

When the monoid $M$ is commutative, as it is in our case, there is a full functor $\DD M\to \HH M$, which is the identity on objects and carries a morphism $(x,y,z):y\to xyz$ of $\DD M$ to the morphism $(y,xz):y\to xyz$ of $\HH M$. Composition with this functor induces a full embedding of $\HH M$-Mod into $\DD M$-mod, whose image consists of the {\em symmetric $\DD M$-modules}, that is, those satisfying that   $x_*=x^*:\A(y)\to \A(xy)$, for all $x,y\in M$ \cite[Chapter II, 7.15]{leech}. Thus, $\HH M$-modules and symmetric $\DD M$-modules are the same thing.

As a direct inspection shows that, for any $\HH M$-module $\A$,  the cochain complex $C(M,1;\A)$ in \eqref{cm1a} coincides with the {\em standard normalized cochain complex} of $M$ with coefficients in $\A$ by Leech \cite[Chapter II, 2.12]{leech}, next theorem follows.

\begin{proposition}\label{h1l} For any $\HH M$-module $\A$, there are natural isomorphisms
$$H^n(M,1;\A)\cong H^n_{^{_\mathrm{L}}}(M,\A),  \hspace{0.4cm} n=0,1,\ldots \, .$$
\end{proposition}

 We now analyze the complex of $\HH M$-modules $\B^{r}\!(\ZZ M)$ for $r\geq 2$ any integer. By construction,
\begin{quote}
  $\bullet$  {\em $\B^{r}\!(\ZZ M)_0$
is the free $\HH M$-module on the unitary set  consisting of the 0-tuple $$[\ \,],\hspace{0.3cm}\text{ with } \pi[\ \,]=e,$$ which we call the generic $0$-cell of $\B^{r}\!(\ZZ M)$},
\end{quote}

\vspace{0.2cm}
\noindent and, for $n\geq 1$,
$$\B^{r}\!(\ZZ M)_n=\bigoplus\limits_{
p+\sum n_i=n}\overline{\B^{r-1}\!(\ZZ M)}_{n_1}\otimes_{\HH\!M} \cdots
\otimes_{\HH\!M}\overline{\B^{r-1}\!(\ZZ M)}_{n_p}.$$ Since
$\overline{\B^{r-1}\!(\ZZ M)}_0=0$ while, for $n_i\geq 1$,
$\overline{\B^{r-1}\!(\ZZ M)}_{n_i}=\B^{r-1}\!(\ZZ M)_{n_i}$, it follows by
induction on $r$ that

\begin{quote} $\bullet$ $\B^{r}\!(\ZZ M)_n=0$ {\em  for $0<n<r$,}
\end{quote}

\vspace{0.2cm}
\noindent and that, for any $r\leq n$,
$$\B^{r}\!(\ZZ M)_n=\bigoplus\limits_{\scriptsize \begin{array}{c}
n_1,\dots,n_p\geq r-1\\
p+\sum n_i=n
\end{array}}\hspace{-0.6cm} \B^{r-1}\!(\ZZ M)_{n_1}\otimes_{\HH\!M} \cdots
\otimes_{\HH\!M}\B^{r-1}\!(\ZZ M)_{n_p}.$$
Then, if
we denote by $|_{_r}$ the symbol $|$ used for the tensor product in
the construction of $\B^{r}\!(\ZZ M)$ from $\B^{r-1}\!(\ZZ M)$,  by Proposition
\ref{fsmon} and induction, we see that

\begin{quote} $\bullet$ {\em  $\B^{r}\!(\ZZ M)_n$, for  $r\leq n$,
is a free $\HH M$-module generated by the set over $M$ consisting
of all $p$-tuples, which we call {\em generic $n$-cells} of
$\B^{r}\!(\ZZ M)$,
$$\alpha_n=[\alpha_{n_1}|_{_r}\alpha_{n_2}|_{_r}
\cdots|_{_r} \alpha_{n_p}], \hspace{0.4cm} \text{ with }\
\pi\alpha_n= \pi\alpha_{n_1}\cdots
\pi\alpha_{n_p},$$ of generic $n_i$-cells of
$\B^{r-1}\!(\ZZ M)$, such that $n_i\geq r-1$ and $p+\sum n_i=n$, with the
relations $\alpha_n=0$ whenever some $\alpha_{n_i}=0$.}
\end{quote}

\vspace{0.2cm}Let us stress that a generic $n$-cell $\alpha_n$ of any  $\B^{r}\!(\ZZ M)$
is actually a  generator of the abelian group
$\B^{r}\!(\ZZ M)_{n}(\pi\alpha_n)$. Indeed, for each $x\in M$,
$\B^{r}\!(\ZZ M)_{n}(x)$ is the free abelian group generated by the elements
$u_*\alpha_n$ with $u$ an element of $M$ and the $\alpha_n$ any
non-zero generic $n$-cell of $\B^{r}\!(\ZZ M)$ such that
$u\,\pi\alpha_n=x$.  Arbitrary elements of the groups $\B^{r}\!(\ZZ M)_{n}(x)$, are referred as
{\em $n$-chains of $\B^{r}\!(\ZZ M)$}.

For any $r\geq 1$,  the  multiplication
$\circ_{^r}$ of $\B^{r}\!(\ZZ M)$ is given by the morphism
 of $\HH M$-modules $$
\circ_{^r}:{\B^{r}\!(\ZZ M)_{n}\otimes_{\HH\!M} \B^{r}\!(\ZZ M)_{m}\to
\B^{r}\!(\ZZ M)_{n+m}}$$ which, according to Proposition \ref{adfu}, are determined on generic cells
by the shuffle product
$$
[\alpha_{n_1}|_{_r}\cdots|_{_r}\alpha_{n_{p}}] \circ_{^r}
[\alpha_{n_{p+1}}|_{_r}\cdots|_{_r}\alpha_{n_{p+q}}]=\sum_\sigma (-1)^{e(\sigma)}
[\alpha_{n_{\sigma^{\scriptsize {\text -1}}(1)}}|_{_r}\cdots|_{_r}
\alpha_{n_{\sigma^{\scriptsize {\text -1}}(p+q)}}],
$$
where the sum is taken over all $(p,q)$-shuffles $\sigma$ and $
e(\sigma)=\sum (1+n_i)(1+n_{p+j})$ summed over all pairs $(i,p+j)$
such that $\sigma(i)>\sigma(p+j)$. In particular, for  $r=1$,
\begin{equation}\label{sp1}[x_1|\cdots|x_n]\circ_{^1} [x_{n+1}|\cdots|x_{n+m}]=\sum_\sigma
(-1)^{e(\sigma)} [x_{\sigma^{\scriptsize {\text -1}}(1)}|\cdots| x_{\sigma^{\scriptsize {\text -1}}(n+m)}],
\end{equation}
where the sum is taken over all $(n,m)$-shuffles $\sigma$ and
$e(\sigma)$ is the sign of the shuffle.

Then, for $r\geq 2$,
\begin{quote} $\bullet$ {\em the boundary $\partial:\B^{r}\!(\ZZ M)_n\to
\B^{r}\!(\ZZ M)_{n-1}$ is the morphism of $\HH M$-modules recursively
defined, on any generic $n$-cell
$\alpha_n=[\alpha_{n_1}|_{_r}\cdots|_{_r}\alpha_{n_{p}}]$ of
$\B^{r}\!(\ZZ M)$ with $\pi\alpha_n=x$ and $\pi\alpha_{n_i}=x_i$, by
\begin{align} \nonumber
\partial_x \alpha_n=&\
 -\sum_{i=1}^{p}(-1)^{e_{i-1}}[\alpha_{n_1}|_{_r}\cdots|_{_r}\alpha_{n_{i-1}}|_{_r}
 \partial_{x_i}\alpha_{n_i}
|_{_r}\alpha_{n_{i+1}}|_{_r}\cdots|_{_r}\alpha_{n_p}]\\
\nonumber
&+\sum_{i=1}^{p-1}(-1)^{e_i}[\alpha_{n_1}|_{_r}\cdots|_{_r}\alpha_{n_{i-1}}|_{_r}
\alpha_{n_i}\!
\circ_{^{\text{r-}1}}\! \alpha_{n_{i+1}}|_{_r}\alpha_{n_{i+2}}|_{_r}\cdots|_{_r}\alpha_{n_p}],
\end{align}
where the exponents $e_i$ of the signs are $e_i=i+\sum n_i$.}
\end{quote}
In the above formula, the term $\partial_{x_i}\alpha_{n_i}$, which refers to the
differential of $\alpha_{n_i}$ in $\B^{r-1}\!(\ZZ M)$, or $\alpha_{n_i}\!
\circ_{^{\text{r-}1}}\! \alpha_{n_{i+1}}$,  is not in general a generic cell of
$\B^{r-1}\!(\ZZ M)$ but a chain; the term is to be expanded by linearity.

Recall now that we have the embedding suspensions  \eqref{sus},
$\s:\B^{r-1}\!(\ZZ M)\hookrightarrow \B^{r}\!(\ZZ M)$, through which we
identify any generic $(n-1)$-cell $\alpha_{n-1}$ of $\B^{r-1}\!(\ZZ M)$ with
the generic $n$-cell $\s\alpha_{n-1}=[\alpha_{n-1}]$ of $\B^{r}\!(\ZZ M)$.
 Hence, by induction, one proves that
 any generic $n$-cell of any $\B^{r}\!(\ZZ M)$
   can be uniquely written in the form
$$
\alpha_n=[x_1|_{_{k_1}}x_2 |_{_{k_2}}\cdots|_{_{k_{m-1}}}x_m]
$$
with $x_i\in M$, $1\leq m$,  $1\leq k_i\leq r$, and
$r+\sum_{i=1}^{m-1}k_i=n$. So written, we have
$\pi\alpha_n=x_1\cdots x_m$, and $\alpha_n=0$ if  $x_i=e$
 for some $i$. Observe that if some $k_i=r$, then $n\geq 2r$. Indeed, the
generic $n$-cells of lowest $n$ appearing in $\B^{r}\!(\ZZ M)$  but not in
$\B^{r-1}\!(\ZZ M)$ are those generic $2r$-cells of the form
$[x_1|_{_{r}}x_2]$. Thus, via the suspension morphism,
$\B^{r-1}\!(\ZZ M)_{n-1}$ is identified with $\B^{r}\!(\ZZ M)_n$ for $r\leq n<2r$,
while $\B^{r-1}\!(\ZZ M)_{n-1}\varsubsetneq \B^{r}\!(\ZZ M)_n$ for $n\geq 2r$. In particular,
we have the commutative diagram of suspensions
$$
 \xymatrix@C=16pt{  \B(\ZZ M)_4\ar[r]\ar@{^(->}[d]_{\s}& \B(\ZZ M)_3
 \ar[r]\ar@{^(->}[d]_{\s}&\B(\ZZ M)_2\ar[r]\ar@{=}[d]_{\s}&\B(\ZZ M)_1
 \ar@{=}[d]_{\s}\ar[r]&\B(\ZZ M)_0\ar[d] \\
 \B^{2}\!(\ZZ M)_5\ar[r]\ar@{^(->}[d]_{\s}& \B^{2}\!(\ZZ M)_4\ar[r]\ar@{=}[d]_{\s}&
 \B^{2}\!(\ZZ M)_3\ar[r]\ar@{=}[d]_{\s}&\B^{2}\!(\ZZ M)_2\ar[r]\ar@{=}[d]_{\s}&0 \\
  \B^{3}\!(\ZZ M)_6\ar[r]\ar@{=}[d]_{\s^{r-3}}& \B^{3}\!(\ZZ M)_5\ar[r]\ar@{=}[d]_{\s^{r-3}}&
  \B^{3}\!(\ZZ M)_4\ar[r]\ar@{=}[d]_{\s^{r-3}}&\B^{3}\!(\ZZ M)_3\ar@{=}[d]_{\s^{r-3}}\ar[r]&0
   \\
   \B^{r}\!(\ZZ M)_{r+3}\ar[r]& \B^{r}\!(\ZZ M)_{r+2}\ar[r]&\B^{r}\!(\ZZ M)_{r+1}\ar[r]&\B^{r}\!(\ZZ M)_r\ar[r]&0
 }
$$
where in the bottom row is $r\geq 3$, and

\vspace{0.2cm}\begin{quote}
 $\bullet$ {\em $\B^{2}\!(\ZZ M)_4$ is the free $\HH M$-module on the set
 of suspensions of the non-zero generic $3$-cells
 $[x_1|x_2|x_3]$  of $\B(\ZZ M)$ together the  non-zero generic
 $4$-cells
$$[x_1|\!|x_2],$$
with $\pi [x_1|\!|x_2]=x_1x_2$, and whose differential is ($x=x_1x_2$)
 \begin{align}\nonumber
  \partial_x[x_1|\!|x_2]&=[x_1|x_2]-[x_2|x_1].
\end{align}}
\end{quote}

\begin{quote}
$\bullet$ {\em $\B^{2}(\ZZ M)_5$ is  the free $\HH M$-module on the set
  of suspensions of the non-zero generic $4$-cells
 $[x_1|x_2|x_3|x_4]$  of $\B(\ZZ M)$ together the non-zero generic
 $5$-cells
$$[x_1|\!|x_2|x_3],\ [x_1|x_2|\!|x_3],$$
 with $\pi [x_1|\!|x_2|x_3]=x_1x_2x_3=\pi[x_1|x_2|\!|x_3]$, and
whose differential is ($x=x_1x_2x_3$)
\begin{align}\nonumber
 \partial_x[x_1|\!|x_2|x_3]=&-x_{2*}[x_1|\!|x_3]+[x_1|\!|x_2x_3]-x_{3*}[x_1|\!|x_2]\\ \nonumber &+[x_1|x_2|x_3]-[x_2|x_1|x_3]+[x_2|x_3|x_1],
\end{align}
\begin{align}\nonumber
 \partial_x[x_1|x_2|\!|x_3]=&-x_{1*}[x_2|\!|x_3]+[x_1x_2|\!|x_3]-x_{2*}[x_1|\!|x_3]\\ \nonumber &-[x_1|x_2|x_3]+[x_1|x_3|x_2]-[x_3|x_1|x_2].
\end{align}}
\end{quote}

\begin{quote}
$\bullet$ {\em $\B^{3}\!(\ZZ M)_6$ is the free $\HH M$-module on the set
  of double suspensions of the non-zero generic $4$-cells
 $[x_1|x_2|x_3|x_4]$  of $\B(\ZZ M)$, together with the suspensions
 of the non-zero generic $5$-cells $[x_1|\!|x_2|x_3]$ and
 $[x_1|x_2|\!|x_3]$ of $\B^{2}\!(\ZZ M)$,  and the  non-zero generic
 6-cells
$$[x_1|\!|\!|x_2],$$
with $\pi[x_1|\!|\!|x_2]=x_1x_2$, whose differential is ($x=x_2x_2$)
 \begin{align}\nonumber
  \partial_x[x_1|\!|\!|x_2]=-[x_1|\!|x_2]-[x_2|\!|x_1].
\end{align}}
\end{quote}

Therefore, from Proposition \ref{adfu}, we get the following.

\begin{theorem}For any $\HH M$-module $\A$, the cohomology groups $H^n(M,r;\A)$, for $n\leq r+2$, are isomorphic to the cohomology groups of
the truncated {\em cochain complexes of normalized $r$th level cochains of $M$ with values in $\A$}, $C(M,r;\A)$,
\begin{equation}\label{cmra}\xymatrix@C=14pt@R=18pt{C(M,r;\A):\hspace{0.3cm} 0\ar[r]& C^0(M,r;\A)\ar[r]&0\longrightarrow\cdots\longrightarrow 0\ar[r]&C^r(M,r;\A)\ar[lld]\\ &
C^{r+1}(M,r;\A)\ar[r]&C^{r+2}(M,r;\A)\ar[r]&C^{r+3}(M,r;\A)
}
\end{equation}
where $C^0(M,r;\A)=\A(e)$, and the remaining non-trivial parts occur in the commutative diagram
\begin{equation}\label{dcs}\begin{array}{c}
 \xymatrix@C=16pt@R=18pt{ 0\ar[r]&C^1(M,1;\A)\ar[r]\ar@{=}[d]&C^2(M,1;\A)\ar[r]\ar@{=}[d]&C^3(M,1;\A)\ar[r]&C^4(M,1;\A)\\
 0\ar[r]&C^2(M,2;\A)\ar[r]\ar@{=}[d]&C^3(M,2;\A)\ar[r]\ar@{=}[d]&
 C^4(M,2;\A)\ar@{->>}[u]_{\s^*}\ar[r]\ar@{=}[d]&C^5(M,2;\A)\ar@{->>}[u]_{\s^*}\\
 0\ar[r]&C^3(M,3;\A)\ar[r]\ar@{=}[d]&C^4(M,3;\A)\ar[r]\ar@{=}[d]&C^5(M,3;\A)\ar[r]\ar@{=}[d]
 &C^6(M,3;\A)\ar@{=}[d]\ar@{->>}[u]_{\s^*}\\
 0\ar[r]&C^r(M,r;\A)\ar[r]&C^{r+1}(M,r;\A)\ar[r]&C^{r+2}(M,r;\A)\ar[r]&C^{r+3}(M,r;\A)
 }\end{array}
\end{equation}
where in the bottom row is $r\geq 3$, and

\vspace{0.2cm}
$\bullet$ $C^4(M,2;\A)$ is the abelian group, under pointwise addition, of pairs  of functions $(g,\mu)$, where
 $$
 \xymatrix{g:M^3\to \bigcup_{x\in M}\A(x)&\mu:M^2\to \bigcup_{x\in M}\A(x),}
 $$
 with $g(x,y,z)\in \A(xyz)$ and $\mu(x,y)\in \A(xy)$, which are normalized in the sense that they take the value 0 whenever some of their arguments are equal to the unit $e$ of the monoid.

 \vspace{0.2cm} $\bullet$ The coboundary $\partial:C^3(M,2;\A)=C^2(M,1;\A)\to C^4(M,2;\A)$ acts on a normalized $2$-cochain $f$ of $M$ in $\A$ by $\partial f=(g,\mu)$, where
\begin{align}\nonumber
 g(x,y,z)&=-x_*f(y,z)+f(xy,z)-f(x,yz)+z_*f(xy),\\ \nonumber
 \mu(x,y)&=f(x,y)-f(y,x).
 \end{align}

 $\bullet$ $C^5(M,2;\A)$ is the abelian group of triples  $(h,\gamma,\delta)$ consisting of normalized functions
 $$
 \xymatrix{h:M^4\to \bigcup_{x\in M}\A(x)&\gamma,\delta:M^3\to \bigcup_{x\in M}\A(x),}
 $$
 with $h(x,y,z,t)\in \A(xyzt)$ and $\gamma(x,y,z),\delta(x,y,z)\in \A(xyz)$.

\vspace{0.2cm}
 $\bullet$  The coboundary $\partial:C^4(M,2;\A)\to C^5(M,2;\A)$ acts on a $2$nd level $4$-cochain $(g,\mu)$  by $\partial(g,\mu)=(h,\gamma,\delta)$, where
\begin{align}\nonumber  h(x,y,z,t)&=
-x_*g(y,z,t)+g(xy,z,t)-g(x,yz,t)+g(x,y,zt)-
t_*g(x,y,z),\\ \nonumber
\gamma (x,y,z)&= -y_*\mu(x,z)+\mu(x,yz)-z_*\mu(x,y)+g(x,y,z)-g(y,x,z)+g(y,z,x),\\  \nonumber
\delta (x,y,z)&= -x_*\mu(y,z)+\mu(xy,z)-y_*\mu(x,z)-g(x,y,z)+g(x,z,y)-g(z,x,y).
  \end{align}

 $\bullet$  $C^6(M,3;\A)$ is the abelian group of quadruples  $(h,\gamma,\delta,\xi)$ consisting of normalized functions
 $$
 \xymatrix@C=5pt{h:M^4\to \bigcup_{x\in M}\A(x),&\gamma,\delta:M^3\to \bigcup_{x\in M}\A(x),&\xi:M^2\to \bigcup_{x\in M}\A(x),}
 $$
 with $h(x,y,z,t)\in \A(xyzt)$, $\gamma(x,y,z),\delta(x,y,z)\in \A(xyz)$, and $\xi(x,y)\in \A(xy)$.

\vspace{0.2cm}
 $\bullet$  The coboundary $\partial:C^5(M,3;\A)=C^4(M,2;\A)\to C^6(M,3;\A)$ acts on a $3$rd-level $5$-cochain by $\partial(g,\mu)=(h,\gamma,\delta,\xi)$, where
\begin{align}\nonumber  h(x,y,z,t)&=
x_*g(y,z,t)-g(xy,z,t)+g(x,yz,t)-g(x,y,zt)+
t_*g(x,y,z),\\ \nonumber
\gamma (x,y,z)&= y_*\mu(x,z)-\mu(x,yz)+z_*\mu(x,y)-g(x,y,z)+g(y,x,z)-g(y,z,x),\\  \nonumber
\delta (x,y,z)&= x_*\mu(y,z)-\mu(xy,z)+y_*\mu(x,z)+g(x,y,z)-g(x,z,y)+g(z,x,y)\\ \nonumber
\xi(x,y)&=-\mu(x,y)-\mu(y,x).
  \end{align}

\end{theorem}
The following corollaries follow directly from the form of the cochain complex \eqref{cmra} and the commutativity of the diagram \eqref{dcs}.

\begin{corollary} For any $r\geq 1$, $H^0(M,r;\A)\cong \A(e)$.
\end{corollary}
\begin{corollary} For any $0<n<r$, $H^n(M,r;\A)=0$.
\end{corollary}
\begin{corollary}\label{hrh1} For any $r\geq 2$, $H^r(M,r;\A)\cong H^1(M,1;\A)$.
\end{corollary}
\begin{corollary}\label{c58} For any $r\geq 2$, $H^{r+1}(M,r;\A)\cong H^3(M,2;\A)$, and there is a natural monomorphism $H^3(M,2;\A)\hookrightarrow H^2(M,1;\A)$.
\end{corollary}
\begin{corollary}\label{c59} For any $r\geq 3$, $H^{r+2}(M,r;\A)\cong H^5(M,3;\A)$,  and there is a natural monomorphism $H^5(M,3;A)\hookrightarrow H^4(M,2;\A)$.
\end{corollary}

Let us now recall that the so-called Grillet cohomology groups  $H^n_{^\mathrm{G}}(M,\A)$, for $1\leq n\leq 3$,  can be computed as the cohomology groups of the truncated
cochain complex $C_{^{_\mathrm{G}}}(M,\A)$,
%\begin{equation}\label{sccomplex}
$$ 0\to
C^1_{^\mathrm{G}}(M,\A)\overset{\delta^1}\longrightarrow
C^2_{^\mathrm{G}}(M,\A)\overset{\delta^2} \longrightarrow
 C^3_{^\mathrm{G}}(M,\A)\overset{\delta^3}\longrightarrow C^4_{^\mathrm{G}}(M,\A),
 $$
%%\end{equation}
called the complex of (normalized on $e\in M$) {\em symmetric cochains}
on $M$ with values in $\A$  \cite[Chapters XII, XIII,
XIV]{grillet}, where

\vspace{0.2cm}$\bullet$ $C^1_{^\mathrm{G}}(M,\A)$ consists of normalized functions $f:M\to \bigcup_{x\in M}\A(x)$, with  $f(x)\in \A(x)$.

\vspace{0.2cm} $\bullet$ $C^2_{^\mathrm{G}}(M,\A)$ consists of normalized functions $f:M^2\to \bigcup_{x\in M}\A(x)$,
with $f(x,y)\in \A(xy)$, such that
$f(x,y)=f(y,x)$.

\vspace{0.2cm} $\bullet$ $C^3_{^\mathrm{G}}(M,\A)$ consists of normalized functions
$f:M^3\to \bigcup_{x\in M}\A(x)$ with $f(x,y,z)\in \A(xyz)$, such
that
$$f(x,y,z)+f(z,y,x)=0,\hspace{0.3cm}
f(x,y,z)+f(y,z,x)+f(z,x,y)=0.$$

\vspace{0.2cm} $\bullet$ $C^4_{^\mathrm{G}}(M,\A)$ consists of normalized functions $f:M^4\to \bigcup_{x\in M}\A(x)$ with $f(x,y,z,t)\in
\A(xyzt)$, such that
$$\begin{array}{l} f(x,y,y,x)=0,\hspace{0.3cm}f(t,z,y,x)+f(x,y,z,t)=0,\\
f(x,y,z,t)-f(y,z,t,x)+f(z,t,x,y)-f(t,x,y,z)=0,\\
f(x,y,z,t)-f(y,x,z,t)+f(y,z,x,t)-f(y,z,t,x)=0.
\end{array}$$

\vspace{0.2cm} $\bullet$ the
coboundary homomorphisms are defined by
$$
\left\{\begin{array}{l}
(\delta^1 f)(x,y)= -x_*f(y)+f(xy)-y_*f(x),\\[4pt]
(\delta^2 f)(x,y,z)= -x_*f(y,z)+f(xy,z)-f(x,yz)+z_*f(x,y), \\[4pt]
 (\delta^3 f)(x,y,z,t)= -x_*f(y,z,t)+f(xy,z,t)-f(x,yz,t)+f(x,y,zt)-
t_*f(x,y,z).
\end{array}
\right.
$$

There is natural injective cochain map \begin{equation}\label{ingr}\begin{array}{l}
\xymatrix{0\ar[r]&
C^1_{^\mathrm{G}}(M,\A)\ar[r]^{\delta^1}\ar@{=}[d]_{i_1=id}&
C^2_{^\mathrm{G}}(M,\A)\ar[r]^{\delta^2} \ar@{_{(}->}[d]_{i_2}&
C^3_{^\mathrm{G}}(M,\A)\ar[r]^{\delta^3}
\ar@{_{(}->}[d]_{i_3}& C^4_{^\mathrm{G}}(M,\A)\ar@{_{(}->}[d]_{i_4}\\
0\ar[r]&  C^3(M,3;\A)\ar[r]^{\partial^3}&
C^4(M,3;\A)\ar[r]^{\partial^4}&C^5(M,3;\A)\ar[r]^{\partial^5}&
C^6(M,3;\A), }
\end{array}
\end{equation}
which is the identity map, $i_1(f)=f$, on symmetric 1-cochains, the
map $i_2(f)=-f$ on symmetric 2-cochains, and on symmetric 3- and 4-cochains is defined
by the simple formulas $i_3(f)=(f,0)$ and $i_4(f)=(-f,0,0,0)$,
respectively. The only non-trivial verification here concerns the
equality $\partial^5i_3=i_4\delta^3$, that is,
$\partial^5(f,0)=(-\delta^3f,0,0,0)$, for any $f\in C^3_{^\mathrm{G}}(M,\A)$,
but it easily follows from Lemma \ref{eleob} below.

\begin{lemma}\label{eleob} Let  $f:M^3\to \bigcup_{x\in M}\A(x)$ be a function with
$f(x,y,z)\in \A(xyz)$. Then $f$ satisfies the symmetry conditions
\begin{equation}\label{easc1}f(x,y,z)+f(z,y,x)=0,\hspace{0.3cm}
f(x,y,z)+f(y,z,x)+f(z,x,y)=0,\end{equation}
if and only if it satisfies either \eqref{easc2} or \eqref{easc3}
below.
\begin{equation}\label{easc2}
f(x,y,z)-f(y,x,z)+f(y,z,x)=0
\end{equation}
\begin{equation}\label{easc3}
f(x,y,z)-f(x,z,y)+f(z,x,y)=0
\end{equation}
\end{lemma}
\begin{proof} The implications $\eqref{easc1}\Rightarrow \eqref{easc2}$ and
$\eqref{easc1}\Rightarrow \eqref{easc3}$ are easily seen. To see
that $\eqref{easc2}\Rightarrow \eqref{easc1}$, observe that, making
the permutation $(x,y,z)\mapsto (z,y,x)$, equation \eqref{easc2} is
written as  $f(y,z,x)=f(z,y,x)+f(y,x,z)$. If we carry this to
\eqref{easc2}, we obtain
$$
f(x,y,z)-f(y,x,z)+f(z,y,x)+f(y,x,z)=f(x,y,z)+f(z,y,x)=0,
$$
that is, the first condition in \eqref{easc1} holds. But then, we
get also the second one simply by replacing the term $f(y,x,z)$ with
$-f(z,x,y)$ in \eqref{easc2}. The proof that
$\eqref{easc3}\Rightarrow \eqref{easc1}$ is parallel.
\end{proof}

\begin{proposition}\label{comgc} For any $\HH M$-module $\A$, the injective cochain map \eqref{ingr} induces natural isomorphisms
$$\xymatrix@C=10pt{
 H^1_{^\mathrm{G}}(M,\A)\cong H^1(M,1;\A),
&
 H^2_{^\mathrm{G}}(M,\A)\cong H^3(M,2;\A),
}$$
and a natural monomorphism
$$
H^3_{^\mathrm{G}}(M,\A)\hookrightarrow H^5(M,3;\A)
.$$
\end{proposition}

\begin{proof}
From diagram \eqref{ingr}, it follows directly that $\ker{\delta^1}=\ker{\partial^3}$ and $i_2\mathrm{Im}\,{\delta^1}=\mathrm{Im}\,{\partial^3}$. Further, $i_2\ker{\delta^2}=\ker{\partial^4}$, since the condition $\partial^4f=0$ on a cochain $f\in C^4(M,3;\A)=C^2(M,1;\A)$ implies the symmetry condition $f(x,y)=f(y,x)$. Then,
$$
H^1_{^\mathrm{G}}(M,\A)=\ker{\delta^1}=\ker{\partial^3}\cong H^3(M,3;\A)\cong H^1(M,1;\A),
$$
and
$$
H^2_{^\mathrm{G}}(M,\A)=\frac{\textstyle \ker{\delta^2}}{\textstyle \mathrm{Im}\,{\delta^1}}\cong
\frac{\textstyle i_2\ker{\delta^2}}{\textstyle i_2\mathrm{Im}\,{\delta^1}}=
\frac{\textstyle \ker{\partial^4}}{\textstyle \mathrm{Im}\,{\partial^3}}\cong H^4(M,3;\A)\cong H^3(M,2;\A).
$$

To prove that the induced homomorphism  $H^3_{^\mathrm{G}}(M,\A)\to H^5(M,3;\A)$ is injective, suppose $f\in
C^3_{^\mathrm{G}}(M,\A)$ is a symmetric 3-cochain such that $i_3f=\partial^4g$ for some $g\in C^4(M,3;\A)=C^2(M,1;\A)$. This means that
the equalities
$$ f(x,y,z)= x_*g(y,z)-g(xy,z)+g(x,yz)-z_*g(x,y),\hspace{0.3cm}0=g(x,y)-g(y,x),$$
hold. Then, $g\in  C^2_{^\mathrm{G}}(M,\A)$ is a symmetric 2-cochain, and
$f=-\delta^2g$ is actually a symmetric
2-coboundary. It follows that the injective map $i_3:
\ker{\delta^3}\hookrightarrow  \ker{\partial^5}$ induces a injective
map in cohomology $H^3C_{^{_\mathrm{G}}}(M,\A)\big)\hookrightarrow H^5
C(M,3;\A)$, as required.
\end{proof}

To complete the list of relationships between the cohomology groups $H^n(M,r;\A)$ with those already known in the literature, let us note that a direct comparison of the cochain complex \eqref{cmra} with the cochain complex in \cite[(6)]{c-c-1}, which computes the lower {\em commutative cohomology groups} $H^n_{^{_\mathrm{ C}}}(M,\A)$, gives the following.

\begin{proposition}\label{p512}
 For any $\HH M$-module $\A$, there are  natural isomorphisms
$$\xymatrix@C=10pt{H^{1}(M,1;\A)\cong  H^1_{^{_\mathrm{C}}}(M,\A),& H^{3}(M,2;\A)\cong  H^2_{^{_\mathrm{C}}}(M,\A), & H^{4}(M,2;\A)\cong  H^3_{^{_\mathrm{C}}}(M,\A).}$$
\end{proposition}

\section{Cohomology classification of symmetric monoidal abelian groupoids}\label{sect6}
This section  is dedicated to showing a precise
classification for symmetric monoidal abelian groupoids, by means of the 3rd level cohomology groups of commutative monoids $H^5(M,3;\A)$.

Symmetric monoidal categories have been studied extensively in the literature
and we refer to Mac Lane \cite{Mac2} and Saavedra \cite{Saa} for the background.  Recall that a {\em groupoid} is a small category all whose morphisms are
invertible. A groupoid $\M$ is said to be {\em abelian} if its isotropy (or vertex) groups
$\mathrm{Aut}_\M(x)$, $x\in \mathrm{Ob}\M$, are all abelian. We will use additive notation for abelian groupoids. Thus, the identity morphism of an object $x$ of an abelian groupoid
$\M$ will be denoted by $0_x$; if $a:x\to y$, $b:y\to z$ are
morphisms,  their composite is written by $b+a:x\to z$, while the
inverse of $a$ is $-a:y\to x$.

\vspace{0.1cm}A {\em symmetric monoidal abelian groupoid}
$$\M=(\M, \otimes, \I,
\boldsymbol{a},\boldsymbol{l},\boldsymbol{r},\boldsymbol{c})$$
consists of an abelian groupoid $\M$, a functor $ \otimes:
\M\times\M\to\M$ (the {\em tensor product}), an object $\I$ (the
{\em unit object}), and  natural isomorphisms
$\boldsymbol{a}_{x,y,z}:(x\otimes y)\otimes z \to x\otimes(y\otimes
z)$, $\boldsymbol{l}_{x}:\I \otimes x \to x$, $\boldsymbol{r}_{x}:x \otimes \I \to x$ (called the {\em associativity} and {\em unit constraints}, respectively) and $\boldsymbol{c}_{x,y}: x\otimes
y \to y\otimes x$ (the {\em symmetry}), such that the four
coherence conditions below hold.
\begin{align}\label{eq5.1}
&\boldsymbol{a}_{{x,y,z\otimes t}}+\boldsymbol{a}_{{x\otimes y,z,t}}=(0_x\!\otimes\!
\boldsymbol{a}_{{y,z,t}})+ \boldsymbol{a}_{{x,y\otimes z,t}}+
(\boldsymbol{a}_{{x,y,z}}\!\otimes\! 0_t),\\
&\label{eq5.2} (0_x\!\otimes\! \boldsymbol{l}_{y})+\boldsymbol{a}_{{x,I,y}}= \boldsymbol{r}_{x}
\!\otimes \!0_y,\\
&\label{eq5.3} (0_y\!\otimes\!\boldsymbol{c}_{x,z})+\boldsymbol{a}_{y,x,z}+
(\boldsymbol{c}_{x,y}\!\otimes\! 0_z)=\boldsymbol{a}_{y,z,x}+\boldsymbol{c}_{x,y\otimes
z}+\boldsymbol{a}_{x,y,z},\\
&\label{eq5.4}  \boldsymbol{c}_{y,x}+\boldsymbol{c}_{x,y}=0_{x\otimes y}.
\end{align}

For further use, we recall that, in any symmetric monoidal abelian
groupoid $\M$, the equalities below hold (see
\cite[Propositions 1.1 and 2.1]{j-s}).
\begin{align}\label{eqad1}
&\boldsymbol{l}_{x\otimes y}+\boldsymbol{a}_{\I,x,y}=\boldsymbol{l}_{x}\otimes 0_y,\hspace{0.3cm}
0_x\otimes \boldsymbol{r}_{y}+\boldsymbol{a}_{x,y,\I}=\boldsymbol{r}_{x\otimes y},
\\
\label{eqad2}
&\boldsymbol{l}_{x}+\boldsymbol{c}_{x,\I}=\boldsymbol{r}_{x},\hspace{0.3cm}
\boldsymbol{r}_{x}+\boldsymbol{c}_{\I,x}=\boldsymbol{l}_{x}.
\end{align}

\begin{example}[{\em $2$-dimensional crossed products}] \label{2dcp}{\em
Every 3rd level 5-cocycle  $(g,\mu)\in
Z^5(M,3;\A)$, gives rise to a symmetric monoidal abelian groupoid
\begin{equation*} \A\rtimes_{g,\mu}\!M=(\A\rtimes_{g,\mu}\!M,\otimes,\mathrm{I},\boldsymbol{a},\boldsymbol{l},\boldsymbol{r}, \boldsymbol{c}),
\end{equation*}
that should  be thought of as a {\em $2$-dimensional crossed product
of $M$ by $\A$}, and it is built as follows:  Its underlying
groupoid has as set of objects the set $M$; if $x\neq y$ are
different elements of the monoid $M$, then there are no morphisms in
$\A\rtimes_{g,\mu}M$ between them, whereas its isotropy group at any
$x\in M$ is $\A(x)$.

The tensor product $\otimes:(\A\rtimes_{g,\mu}\!M)\times
(\A\rtimes_{g,\mu}\!M)\to \A\rtimes_{g,\mu}\!M$ is given on objects  by
multiplication in $M$, so $x\otimes y=xy$, and on
morphisms by the group homomorphisms
%\begin{equation}\label{ot}
$$
\otimes:\A(x)\times \A(y)\to \A(xy), \hspace{0.4cm} a_x\otimes a_y=y_* a_x+ y_*a_x.
$$
%\end{equation}
The unit object is $\I=e$, the unit of the monoid $M$, and the
structure constraints are
\begin{align}\nonumber
\boldsymbol{a}_{x,y,z}=\ &g(x,y,z):(xy)z\to x(yz),\\ \nonumber
\boldsymbol{c}_{x,y}=\ &\mu(x,y):xy\to yx, \\ \nonumber
\boldsymbol{l}_{x}= \ & 0_x:ex=x\to x \\ \nonumber
\boldsymbol{r}_{x}=\ &0_x:xe=x\to x,
\end{align}
which are easily seen to be natural since $\A$ is an abelian group
valued functor.  The coherence conditions $(\ref{eq5.1})$,
$(\ref{eq5.3})$,  and $(\ref{eq5.4})$ follow from the 5-cocycle
condition $\partial^5(h,\mu)=(0,0,0,0)$, while the coherence condition
$(\ref{eq5.2})$ holds due to the normalization conditions
$h(x,e,y)=0$.}
\end{example}

\vspace{0.2cm}If $\M$, $\M'$ are symmetric monoidal abelian groupoids, then a
{\em symmetric monoidal functor} $
F=(F,\varphi,\varphi_0):\M\to \M'$ consists of a functor between the
underlying groupoids $F:\M\to \M'$, natural isomorphisms $
\varphi_{x,y}:Fx\otimes' Fy\to F(x\otimes y)$, and an isomorphism
$\varphi_0:\I'\to F\I $, such that the following coherence
conditions hold:
\begin{align}\label{eq5.5}
&\varphi_{x,y\otimes z}+(0_{Fx}\!\otimes'\! \varphi_{y,z})+\boldsymbol{a}'_{Fx,Fy,Fz}=
F(\boldsymbol{a}_{x,y,z})+\varphi_{x\otimes
y,z}+(\varphi_{x,y}\!\otimes'\! 0_{Fz}),
\\
\label{eq5.6}& F(\boldsymbol{l}_x)+\varphi_{\I,x}+( \varphi_0 \!\otimes'\! 0_{Fx})=
\boldsymbol{l}'_{Fx}, \
F(\boldsymbol{r}_x)+\varphi_{x,\I}+(0_{Fx}\!\otimes'\! \varphi_0)=
\boldsymbol{r}'_{Fx}\,,
\\
\label{eq5.7}& \varphi_{y,x} +\boldsymbol{c}'_{Fx,Fy}=F(\boldsymbol{c}_{x,y})+\varphi_{x,y}.
\end{align}
%The symmetric monoidal functor is called {\em strict} when each of the
%isomorphisms $\varphi_{x,y}$ and $\varphi_0$ is an identity.

Suppose $F':\M\to \M'$ is another symmetric monoidal functor. Then, a {\em
symmetric monoidal isomorphism} $\theta:F\Rightarrow F'$ is a natural
isomorphism between the underlying functors, $\theta:F\Rightarrow F'$,
such that the following coherence conditions hold:
\begin{equation*}
\varphi'_{x,y}+(\theta_x\otimes'\theta_y)=\theta_{x\otimes y}+\varphi_{x,y}, \ \
\theta_\I+\varphi_0=\varphi'_0.
\end{equation*}

\begin{comment}
\begin{example} {\em Let $(g,\mu),(g',\mu')\in Z^5(M,3:\A)$
be 3rd level $5$-cocycles of a commutative monoid. Then, any
$4$-cochain $f\in C^4(M,3;\A)=C^2(M,1;\A)$ such that
$(h,\mu)=(h',\mu')+\partial^4g$ induces a symmetric monoidal isomorphism
\begin{equation}\label{isoco}
F(f)=(id,f,0_1):\A\rtimes_{g,\mu}\!M\cong \A\rtimes_{g',\mu'}\!M
\end{equation}
which is the identity functor on the underlying groupoids. Its structure isomorphisms are given by $\varphi_{x,y}=f(x,y):xy \to xy$
and  $\varphi_0=0_e:e\to e$, respectively. Since the groups
$\A(xy)$ are abelian, these isomorphisms $\varphi_{x,y}$ are
natural.  The equality $(g',\mu')=(g,\mu)+\partial^4f$ implies coherence conditions
\eqref{eq5.5} and \eqref{eq5.7}, while
the conditions in \eqref{eq5.6} trivially hold because of the
normalization conditions  $f(x,e)=0_{x}$.

Moreover, any $f^1\in C^3(M,3;\A)=C^1(M,1;\A)$ 3rd level $3$-cochain verifying
$f'=f+\partial^3f^1\in C^2(M,1;\A)$, induced a symmetric monoidal isomorphism
$ \theta(f^1): F(f)\Rightarrow F(f') $. Indeed, $\theta(f^1)$ is defined
by $\theta(f^1)_x=f^1(x):x\to x$, for each $x\in M$. So defined,
$\theta$ is natural because the groups
$\A(x)$ are abelian; the first  condition in \eqref{eq5.8} holds owing to the
equality $f'=f+\partial^1f^1$, and the second one thanks to the
normalization condition $f^1(e)=0_e$ of $f^1$.}
\end{example}
\end{comment}

With compositions defined in a natural way, symmetric monoidal abelian
groupoids, symmetric monoidal functors, and symmetric monoidal isomorphisms form a
2-category \cite[Chaper V, \S 1]{G-Z}.
 A symmetric monoidal functor  $F:\M\to \M'$ is called a {\em symmetric monoidal
equivalence} if it is an equivalence in this 2-category.

\begin{comment}
\begin{remark}\label{remar}  From the Coherence Theorem
for monoidal categories \cite{Mac2,Mac3}, it follows that
 every symmetric monoidal abelian groupoid is symmetric monoidal
equivalent to a symmetric monoidal  strict one, that is, to one in
which all the structure constraints $\boldsymbol{a}_{x,y,z}$,
$\boldsymbol{l}_x$, and $\boldsymbol{r}_x$ are identities.
\end{remark}
\end{comment}

 Our goal is to show a classification for symmetric monoidal abelian groupoids,
 where two symmetric monoidal abelian groupoids connected by a symmetric monoidal
 equivalence are considered the same, as stated in the theorem below. Recall that any homomorphism of monoids $i:M\to M'$ induces a functor $ \HH M\to \HH M'$ in a obvious way, and then, by composition with it, a functor $i^*: \HH M'\text{-Mod}\to \HH M\text{-Mod}$.

\begin{theorem}[Classification of Symmetric Monoidal Abelian Groupoids]\label{mainthcl}
$(i)$ For any symmetric monoidal abelian groupoid  $\M$, there exist
a commutative monoid $M$, a $\HH M$-module $\A$, a
$3$rd level $5$-cocycle $(g,\mu)\in Z^5(M,3;\A)$, and a symmetric monoidal
equivalence
$$
\A\rtimes_{g,\mu}\!M \simeq \M.
$$

$(ii)$ For any two $3$rd level $5$-cocycles $(g,\mu)\in Z^5(M,3;\A)$ and $(g',\mu')\in Z^5(M',3;\A')$, there is a symmetric monoidal equivalence
$$
\A\rtimes_{g,\mu}\!M \simeq \A'\rtimes_{g'\!,\mu'}\!M'
$$
 if and and
only if there exist an isomorphism of monoids $i:M\cong M'$ and a
natural isomorphism  $\psi:\A\cong i^*\A'$,  such that the equality of
cohomology classes below holds.
$$
[g,\mu]=\psi_*^{-1}i^*[g',\mu']\in  H^5(M,3;\A)
$$
\end{theorem}
\begin{proof} $(i)$  Let $\M=(\M, \otimes, \I,
\boldsymbol{a},\boldsymbol{l}, \boldsymbol{r},\boldsymbol{c})$ be any
given symmetric monoidal abelian groupoid.

By the coherence theorem \cite{Mac2}, there is no loss of generality in assuming that $\M$ is itself strictly unitary, that is, where both  unit constraints $\boldsymbol{l}$ and $\boldsymbol{r}$ are identities. Then, we observe that $\M$ is symmetric monoidal equivalent to another one that is totally disconnected, that is,  where there is no morphism between different objects. Indeed, by the generalized Brandt's theorem \cite[Chapter 6, Theorem 2]{hig},  there is a totally disconnected groupoid, say $\M'$, with an equivalence of groupoids $\M\to \M'$. Hence, by
Saavedra \cite[I,
4.4]{Saa}, we can transport the symmetric monoidal structure
along this equivalence so that $\M'$ becomes a strictly unitary symmetric monoidal abelian groupoid and the equivalence a symmetric monoidal one.

Hence, we assume that $\M$ is totally disconnected and
 strictly unitary. Then,
a triplet $(M,\A,(g,\mu))$, such that $\A\rtimes_{g,\mu}\!M =
\M$ as symmetric monoidal abelian groupoids, can be defined as follows:

\vspace{0.2cm} $\bullet$ {\em The monoid $M$}. Let $M=\mathrm{Ob}\M$ be the
set of objects of $\M$. The tensor
functor $\otimes:\M\times \M\to\M$ determines a multiplication on
$M$, simply by $xy=x\otimes y$, for any $x,y\in M$. Since $\M$ is strictly unitary, this multiplication on $M$ is
unitary with $e=\I$, the unit object of $\M$. Moreover, it is
associative and commutative since $\M$ being totally disconnected implies that $(xy)z=x(yz)$ and
$xy=yx$. Thus, $M$ becomes a commutative monoid.

\vspace{0.2cm} $\bullet$ {\em The $\HH M$-module $\A$}.
For each $x\in M=\mathrm{Ob}\M$, let $\A(x)=\mathrm{Aut}_\M(x)$ be the
vertex group of the underlying abelian groupoid at $x$.
Since the diagrams below commute, due to the
naturality of the structure constraints and the symmetry,
$$
\xymatrix@R=26pt{(xy)z\ar[d]_{(a_x\otimes a_y)\otimes a_z}
\ar[r]^{\boldsymbol{a}_{x,y,z}}&x(yz)\ar[d]^{a_x\otimes (a_y\otimes a_z)}\\
(xy)z\ar[r]^{\boldsymbol{a}_{x,y,z}}&x(yz)
}\hspace{0.2cm}
\xymatrix@R=27pt{xy\ar[d]_{a_x\otimes a_y}
\ar[r]^{\boldsymbol{c}_{x,y}}&yx\ar[d]^{a_y\otimes a_x}\\
xy\ar[r]^{\boldsymbol{c}_{x,y}}&yx
}\hspace{0.2cm}
\xymatrix@R=26pt{ex=x\ar[d]_{0_e\otimes a_x}
\ar[r]^-{0_x}&x\ar[d]^{ a_x}\\
ex=x\ar[r]^-{0_x}&x
}
$$
it follows that the equations below hold.
\begin{equation}\label{acubeha}
(a_x\otimes a_y)\otimes a_z = a_x\otimes ( a_y\otimes a_z),\hspace{0.2cm} a_x\otimes a_y=a_y\otimes a_x,
\hspace{0.2cm} 0_e\otimes a_x=a_x,
\end{equation}

Then, if we write $y_*:\A(x)\to \A(xy)$ for the homomorphism such that
$$y_*a_x=0_y\otimes a_x=a_x\otimes 0_y,$$ the equalities
$$
\begin{array}{l} (yz)_*(a_x)=0_{yz}\otimes a_x=(0_y\otimes 0_z)\otimes a_x\overset{(\ref{acubeha})}=
0_y\otimes (0_z\otimes a_x )=y_*(z_*a_x), \\
e_*a_x=0_e\otimes a_x\overset{(\ref{acubeha})}=a_x,
\end{array}
$$
implies that the assignments $x\mapsto \A(x)$, $(x,y)\mapsto y_*:\A(x)\to
\A(xy)$, define an $\HH M$-module. Observe that this $\HH M$-module $\A$ determines indeed the tensor product $\otimes$ of $\M$,
since
\begin{align}\nonumber a_x\otimes a_y\ =\ & (a_x+0_x)\otimes (0_y+
a_y) = (a_x\otimes 0_y)+ (0_x\otimes a_y)
=\  y_*a_x+x_*a_y.
\nonumber
\end{align}

$\bullet$ {\em The $3$rd level $5$-cocycle $(g,\mu)\in Z^5(M,3;\A)$}. The
associativity constraint  and the symmetry of $\M$  can be
written in the form $\boldsymbol{a}_{x,y,z}=g(x,y,z)$ and
$\boldsymbol{c}_{x,y}=\mu(x,y)$, for some given lists
$\big(g(x,y,z)\in \A(xyz)\big)_{^{x,y,z\in M}}$ and
$\big(\mu(x,y)\in \A(xy)\big)_{^{x,y\in M}}$. Since $\M$ is
strictly unitary, equations in \eqref{eq5.2} and \eqref{eqad1}
give the normalization conditions  $g(x,e,y)=0=g(e,x,y)=g(x,y,e)$
for $g$, while equations in \eqref{eqad2} imply the normalization
conditions  $\mu(x,e)=0=\mu(e,x)$ for $\mu$. Thus, $(g,\mu)\in
C^5(M,3;\A)$ is a 3rd level 5-cochain. By the coherence conditions \eqref{eq5.1},
\eqref{eq5.3}, and \eqref{eq5.4} we have that
\begin{align}\nonumber
&g(x,y,zt)+g(xy,z,y)=x_*g(y,z,y)+ g(x,yz,y)+t_*g(x,y,z)\\ \nonumber
& y_*\mu(x,z)+g(y,x,z)+ z_*\mu(x,y)=g(y,z,x)+\mu(x,yz)+g(x,y,z),\\ \nonumber
& \mu(x,y)+\mu(y,x)=0,
\end{align}
and combining the last two equations we also have
$$
-y_*\mu(z,x)+g(y,x,z)- z_*\mu(y,x)=g(y,z,x)-\mu(yz,x)+g(x,y,z).
$$
Hence, we obtain the required cocycle condition $\partial^3(g,\mu)=(0,0,0)$.
Since a direct comparison shows that
$\M=\A\rtimes_{g,\mu}\!M$  as symmetric monoidal abelian groupoids, the proof of this part is complete.

\vspace{0.2cm} $(ii)$ Suppose there exist an
isomorphism of monoids $i:M\cong M'$ and a natural isomorphism
$\psi:\A\cong i^*\A'$ such that $\psi_*[g,\mu]=i^*[g',\mu']\in
H^5(M,3;i^*\A') $. This  implies that there is a 3rd level 4-cochain
$f\in C^4(M,3;i^*\A')=C^2(M,1;i^*\A')$ such that
\begin{align} \label{qa1}
  \psi_{xyz} g(x,y,z)=& g'(ix,iy,iz)+(ix)_*f(y,z)-f(xy,z)+f(x,yz)-(iz)_*f(x,y),\\[4pt] \label{qa2}
   \psi_{xy}\mu(x,y)=&\mu'(ix,iy)-f(x,y)+f(y,x).
\end{align}

Then, a symmetric monoidal isomorphism
\begin{equation*}\label{fg}F(f)=(F,\varphi,\varphi_0):\A\rtimes_{g,\mu}\!M \to
\A'\rtimes_{g'\!,\mu'}\!M'.\end{equation*}
can be defined as follows: The
underlying functor acts by $F(a_x:x\to x)=(\psi_xa_x:ix\to ix)$.
The constraints of $F$ are given by
$\varphi_{x,y}=f(x,y):(ix)\, (iy) \to i(xy)$, and
$\varphi_0=0_e:e\to ie=e$. So defined, it is easy to see that $F$ is
an isomorphism between the underlying groupoids.
The naturality of the isomorphisms $\varphi_{x,y}$, that is,
\begin{equation}\label{natfi}
 \psi_{xy}(x_*a_y+y_*a_x) + \varphi_{x,y}=\varphi_{x,y}+(ix)_*\psi_{y}a_y+ (iy)_*\psi_{x}a_x
 \end{equation}
for  $a_x\in \A(x)$, $a_y\in\A(y)$, holds owing to the commutativity of
$\A'(i(xy))$ and the naturality of $\psi:\A\cong i^*\A'$, which says
that
\begin{equation}\label{ea3}\psi_{xy}(x_*a_y)=(ix)_*\psi_ya_y.\end{equation}
 The coherence
conditions $(\ref{eq5.5})$ and $(\ref{eq5.7})$ are obtained as a consequence of equations \eqref{qa1} and
\eqref{qa2}, respectively, whereas the conditions in \eqref{eq5.6} trivially follow from the
normalization conditions $f(x,e)=0_{ix}=f(e,x)$.

Conversely, suppose we have
$F=(F,\varphi,\varphi_0):\A\rtimes_{g,\mu}\!M \to
\A'\rtimes_{g'\!,\mu'}\!M'$ a symmetric monoidal equivalence. By
\cite[Lemma 18]{cegarra2}, there is no loss of generality in
assuming that $F$ is strictly unitary in the sense that
$\varphi_0=0_e:e\to e=Fe$. As the underlying functor establishes an
equivalence between the underlying groupoids, and these are
totally disconnected, it is an isomorphism.

We write $i:M\cong M'$ for the bijection established by $F$ between the object sets; that is, such that $ix=Fx$, $x\in M$.
Then, $i$ is actually an isomorphism of monoids, since the existence
of the structure isomorphisms $\varphi_{x,y}: (ix)(iy)\to i(xy)$
implies $(ix)(iy)=i(xy)$.

Let us write now $\psi_x:\A(x) \cong \A'(ix)$ for the isomorphism such that
$Fa_x=\psi_xa_x$, for each automorphism $a_x\in \A(x)$, and  $x\in M$. The
naturality of the automorphisms $\varphi_{x,y}$ tell us that the
equalities \eqref{natfi} hold.
In particular, when $a_x=0_x$, we obtain the equation
\eqref{ea3} and so $\psi:\A\cong i^*\A'$ is indeed a natural
isomorphism.

Finally, if we write $f(x,y)=\varphi_{x,y}$, for each $x,y\in M$, we have a 3rd level $4$-cochain
$
f(F)=\big(f(x,y)\in \A'(i(xy))\big)_{^{x,y\in M}},
$
since the equations $f(x,e)=0_{ix}=f(e,x)$ hold due to \eqref{eq5.6}.
Equations \eqref{qa1} and \eqref{qa2} follow from
to the coherence equations \eqref{eq5.5} and \eqref{eq5.7}. This means that $\psi_*(g,\mu) =
i^*(g',\mu')+\partial^4f$ and, therefore,  we have that
$\psi_*[g,\mu] = i^*[g',\mu']\in  H^5(M,3;i^*\A')$, whence  $[g,\mu]
= \psi_*^{-1}i^*[g',\mu']\in H^5(M,3;\A)$.
\end{proof}

\section{Cohomology of cyclic monoids}\label{sect7}
In this section we compute the cohomology groups $H^n(C,r;\A)$, for $n\leq r+2$, when $C$ is any cyclic monoid. The method we employ follows similar lines to the one used by Eilenberg and Mac Lane in \cite[\S 14 and \S 15]{E-M-II}, for computing higher level cohomology of cyclic groups, though the generalization to monoids is highly nontrivial.
\subsection{Cohomology of finite cyclic monoids}$~$

The structure of finite cyclic monoids was first stated by Frobenius \cite{Frobenius}.
Briefly, let us recall that if $\equiv$ is any not equality congruence on the additive monoid
$\N=\{0,1,\dots\}$ of natural numbers, then the least $m\geq 0$ such that $m\equiv x$ for some $x\neq m$ is called the {\em index} of the congruence, and the least
$q\geq 1$ such that $m\equiv m+q$ is called its {\em period}. Hence,
$$
x\equiv y \text{\ \  if and only if either \ } x=y<m, \text{\ or\ } x,y\geq m \text{\ and\  } x\equiv y \!\!\! \mod{q}.
$$
The quotient $
\N/\equiv
$
is called the {\em cyclic monoid of index $m$ and period $q$}, and denoted here $C_{m,q}$.
As $\N$ is a free monoid on the generator 1, every finite cyclic monoid is isomorphic to a proper quotient of $\N$ and, therefore, to a monoid  $C_{m,q}$ for some $m$ and $q$.

\vspace{0.2cm}
From now on, $C=C_{m,q}$ denotes the finite cyclic monoid of index $m$ and period $q$. We assume that $m + q \geq 2$, so that $C$ is not the zero monoid.
\vspace{0.2cm}

Since every element of $C$ can be written uniquely in the form $[x]$ with $0\leq x<m+q$,  this monoid can be described as the set
$$
C=\{0,1,\dots, m,m+1,\dots, m+q-1\},
$$
with addition
$$
x\oplus y = \wp(x+y),
$$
where $\wp:\N\to C$ is the projection map given by
\begin{equation*}
\wp( x) =\left\{\begin{array}{lll}x&\text{if}&  x<m+q\\
x-kq&\text{if}& m+kq\leq x<m+(k+1)q.
\end{array}\right.
\end{equation*}

To any pair $x,y\in C$, we can associate the useful integer
$$
s(x,y)=\frac{ (x+y)-(x\oplus y)}{ q},
$$
which satisfies $s(x,y)\geq 1$ if $x+y\geq m+q$, whereas $s(x,y)=0$ if $x+y<m+q$. It follows directly from the associativity in $C$ that the cocycle property below holds.
\begin{equation}\label{sisco}
s(y,z)+s(x,y\oplus z)=s(x\oplus y,z)+s(x,y).
\end{equation}

Next, we construct a specific commutative DGA-algebra over $\HH C $, denoted by $$\mathcal{R}=\mathcal{R}(C),$$
which is homologically equivalent to $\B(\ZZ C)$ but algebraically simpler and more lucid.
For each integer $k=0,1,\dots$, let us choose unitary sets over $ C$, $\{ v_k\}$ and $\{ w_k\}$, with
\begin{equation}\label{a'1}\begin{array}{cc}\pi v_k=\wp(km),& \pi w_k=\wp(km+1),\end{array}\end{equation}
and define
\begin{equation}\label{a'2}\left\{
\begin{array}{lcc}
\mathcal{R}_{2k}&=& \text{the free } \HH C\text{-module on } \{v_k\},\\[4pt]
\mathcal{R}_{2k+1}&=& \text{the free } \HH C\text{-module on } \{w_k\}.
\end{array}\right.
\end{equation}
The augmentation $\alpha:\mathcal{R}_0\to \Z$, the differential $\partial:\mathcal{R}_n\to\mathcal{R}_{n-1}$, and the multiplication  $\circ:\mathcal{R}\otimes_{\HH C }\mathcal{R}\to \mathcal{R}$
are determined by the equations
 \begin{equation}\label{difcep}
\begin{array}{lcl}
\alpha v_0=1,&
\partial v_{k+1} = (m+q)\big((m+q-1)_*w_k\big)-m\big((m-1)_*w_k\big),&
\partial w_k= 0,
\end{array}
\end{equation}
\begin{equation}\label{mulep}\begin{array}{lcl}
v_k\circ v_l={\scriptsize \begin{pmatrix}k+l\\k  \end{pmatrix}}v_{k+l},&
w_k\circ w_l=0,  &
v_k\circ w_l={\scriptsize \begin{pmatrix}k+l\\k  \end{pmatrix}}w_{k+l}=w_l\circ v_k,
\end{array}
\end{equation}
 and the unit is $v_0$.
 \begin{proposition} $\mathcal{R}$, defined as above, is a commutative DGA-algebra over $\HH C $.
 \end{proposition}
 \begin{proof}
 By Proposition \ref{adfu}, the mapping in \eqref{difcep},
 $v_{k+1}\mapsto \partial v_{k+1}$,
 determines a morphism of $\HH C $-modules $\partial: \mathcal{R}_{2k+2}\to \mathcal{R}_{2k+1}$ since
 $$\begin{array}{l}
(m+q-1)\oplus \pi w_k \overset{\eqref{a'1}}=(m+q-1)\oplus \wp(km+1)=\wp(m+q+km)=\wp(m+km)=\pi v_{k+1},\\[4pt]
(m-1)\oplus \pi w_k  \overset{\eqref{a'1}}= (m-1)\oplus \wp(km+1)=\wp(m+km)=\pi v_{k+1},
\end{array}
 $$
 and therefore $\partial v_{k+1}\in \mathcal{R}_{2k+1}(\pi v_{k+1})$. Similarly,
 by Proposition \ref{adfu}, we see that the formulas in \eqref{mulep} determine a multiplication morphism
 of $\HH C $-modules since $\wp(km)\oplus \wp(lm)=\wp((k+l)m)$ and $\wp(km)\oplus \wp(lm+1)=\wp((k+l)m+1)$.
 Associativity condition \eqref{dga4} follows  from the equality on combinatorial numbers
 $${\scriptsize \begin{pmatrix}k+ l+t\\k\end{pmatrix}}+ {\scriptsize \begin{pmatrix} l+t\\t\end{pmatrix}}=
 \frac{(k+l+t)!}{k!\,l!\,t!}={\scriptsize  \begin{pmatrix}k+ l+t\\k+l\end{pmatrix}}+ {\scriptsize  \begin{pmatrix} k+l\\l\end{pmatrix}},
 $$
 while condition \eqref{dga5} holds thanks to the equality
 $$
 {\scriptsize \begin{pmatrix} k+l-1\\k-1\end{pmatrix}}+ {\scriptsize  \begin{pmatrix} k+l-1\\k\end{pmatrix}}= {\scriptsize \begin{pmatrix} k+l\\ k\end{pmatrix}},$$
 and the remaining conditions in \eqref{dga1}-\eqref{dga3} are quite obviously verified.
  \end{proof}

In  next proposition we shall define a morphism  $f:\B(\ZZ C)\to \mathcal{R}$.
Previously, observe that the graded $\HH C $-module $\{\mathcal{R}_n\}$
admits another structure of commutative graded algebra over $\HH C $ (although it does not respect
the differential structure), whose multiplication is determined by the simpler formulas
$$\begin{array}{lcl}
v_k\bullet v_l=v_{k+l},&w_k\bullet  w_l=0,&v_k\bullet  w_l=w_{k+l}=w_l\bullet  v_k.
\end{array}$$

\begin{proposition}\label{gep1}   A  morphism $f:\B(\ZZ C)\to \mathcal{R}$, of DGA-algebras over $\HH C $,
 may be defined by the recursive formulas
\begin{equation}\label{deff}\left\{\begin{array}{ccl}f[\ ]&=&v_0,\\ f[x]&=&x((x-1)_*w_0), \\[5pt]
 f[x\mid y]&=&\left\{ \begin{array}{lcl}0&\text{if}&x+y<m+q,\\[4pt] ((x\oplus y)\text{-}m)_*\Big(\sum\limits_{i=0}^{s(x,y)\text{-}1} (iq)_*v_1\Big) &\text{if}& x+y\geq m+q,\end{array}\right.\\[20pt]
 f[x\mid y\mid \sigma]&=& f[x\mid y]\bullet f[\sigma],
 \end{array}\right.
\end{equation}
where $\sigma=[z|\cdots]$ is any cell of dimension 1 or greater.
\end{proposition}
\begin{proof} This is divided into four parts. Note first that, from the inequalities
 $$
m+s(x,y)q\leq (x\oplus y)+s(x,y)q=x+y<2m+2q-1,
$$
 it follows that  $s(x,y)q<m+2q-1$. Therefore, for any $0\leq i<s(x,y)$, we have $iq=\wp(iq)\in C$ and the formula above for $f[x\mid y]$ is well defined.

 {\em Part 1.}  We prove in this step that the assignment in \eqref{deff} extends to a morphism of complexes of
 $\HH C $-modules.
 This follows from  Proposition \ref{adfu}, since
 one verifies recursively that
 $$f[x_1\mid\cdots\mid x_n]\in \mathcal{R}_{n}(x_1\oplus \cdots\oplus x_n)$$
 as follows: The case when $n=0$ is obvious. When $n=1$, it holds since $w_0\in \mathcal{R}_1$ and
 $(x_1-1)\oplus \pi w_0=(x_1-1)\oplus 1=x_1$, and for $n=2$ since $v_1\in \mathcal{R}_2$ and
$$((x_1\oplus x_2)\text{-}m)\oplus \pi v_1= ((x_1\oplus x_2)\text{-}m)\oplus m=x_1\oplus x_2.$$
Then, for $n\geq 3$, induction gives
$$f[x_1|\cdots| x_n]=f[x_1| x_2]\bullet f[x_3|\cdots|x_n]\in \!\mathcal{R}_{2}(x_1\oplus x_2)\bullet \mathcal{R}_{n-2}(x_3\oplus \cdots\oplus x_n)
\subseteq \!
 \mathcal{R}_{n}(x_1\oplus \cdots\oplus x_n).$$

\noindent {\em Part 2.}  We prove now that  $\partial f=f\partial$.

\underline{For a 1-cell} $[x]$ of $\B(\ZZ C)$, we have $\partial f[x]=x((x-1)_*\partial w_0)\overset{\eqref{difcep}}=0= f\partial[x]$.

\underline{For a 2-cell} $[x\mid y]$, we have
$$
f\partial [x\mid y]=x_*f[y]-f[x\oplus y]+y_*f[x].
$$
To compare with $\partial f[x\mid y]$, we shall distinguish three cases:

- {\em Case $x+y<m+q$.} In this case $\partial f[x\mid y]=0$, and also
\begin{align}\nonumber
f\partial [x\mid y]&=y((x+y-1)_*w_0)-(x+y)((x+y-1)_*w_0)+x((x+y-1)_*w_0)=0.
\end{align}

-{\em Case $x+y\geq m+q$ and $x\oplus y=m$.} Here, $(x-1)\oplus y=m+q-1 =x\oplus (y-1)$. Then,
\begin{align*}\nonumber
\partial f[x\mid y]&= \sum\limits_{i=0}^{s(x,y)\text{-}1} (m+q)((iq\oplus (m+q-1))_*w_0)\text{-}m((iq\oplus (m-1))_*w_0)\\
\nonumber &=(m+q)((m+q-1)_*w_0)\text{-}m((m-1)_*w_0)\\ &\hspace{0.4cm}+\sum\limits_{i=1}^{s(x,y)\text{-}1} (m+q)( (m+q-1)_*w_0)\text{-}m((m+q-1))_*w_0)\\
\nonumber &=(m+q)((m+q-1)_*w_0)\text{-}m((m-1)_*w_0)+ (s(x,y)-1)q((m+q-1)_*w_0)\\
\nonumber &=(m+s(x,y)q)((m+q-1)_*w_0)\text{-}m((m-1)_*w_0)\\
\nonumber &=(x+y)((m+q-1)_*w_0)\text{-}m((m-1)_*w_0)=f\partial [x\mid y].
\end{align*}

-{\em Case $x+y\geq m+q$ and $x\oplus y>m$.} In this case, $(x-1)\oplus y=(x\oplus y)-1=x\oplus (y-1)$, whence
\begin{align}\nonumber
\partial f[x\mid y]&=\sum\limits_{i=0}^{s(x,y)\text{-}1} ((x\oplus y)\text{-}m)\oplus iq)_*\partial v_1=
\sum\limits_{i=0}^{s(x,y)\text{-}1} (m+q)\big(((x\oplus y)\text{-}m)\oplus ((iq\oplus (m+q-1))_*w_0\big)
\\ \nonumber &\hspace{0.4cm} -\sum\limits_{i=0}^{s(x,y)\text{-}1}m\big(((x\oplus y)\text{-}m)\oplus(iq\oplus (m-1))_*w_0)
=  \sum\limits_{i=0}^{s(x,y)\text{-}1} (m+q)((x\oplus y)-1)_*w_0) \\ \nonumber &\hspace{0.4cm} -\sum\limits_{i=0}^{s(x,y)\text{-}1}m((x\oplus y)-1))_*w_0)
=qs(x,y)((x\oplus y)-1)_*w_0)\\ \nonumber &
=(y-(x\oplus y)+x)((x\oplus y)-1)_*w_0)=f\partial[x\mid y].
\end{align}

\underline{For a 3-cell} $[x\mid y\mid z]$,  we have to prove that $f\partial [x\mid y\mid z]=0$ or, equivalently, that
\begin{equation}\label{f3-cocy}
x_*f[ y\mid z]+f[x\mid y\oplus z]=z_*f[x\mid y]+f[x\oplus y\mid z].
\end{equation}

Since $x+(y\oplus z)=x\oplus y\oplus z+s(x,y\oplus z)q$, it follows that
$$x\oplus((y\oplus z)\text{-}m)=((x\oplus y\oplus z)\text{-}m)\oplus \wp(s(x,y\oplus z)q),$$
whenever $y\oplus z\geq m$. Then, we can write
\begin{align}\nonumber
x_*f[ y\mid z]&=\left\{\begin{array}{ll}0,&\text{if } s(y,z)=0,\\
(x\oplus((y\oplus z)\text{-}m))_*\Big(\sum\limits_{i=0}^{s(y,z)\text{-}1} \wp(iq)_* v_1\Big),&\text{if } s(y,z)\geq 1,
 \end{array} \right.\\ \nonumber
 &=\left\{\begin{array}{ll}0,&\text{if } s(y,z)=0,\\
((x\oplus y\oplus z)\text{-}m)_*\Big(\sum\limits_{i=0}^{s(y,z)\text{-}1} \wp\big(s(x,y\oplus z)q+iq\big)_* v_1\Big),&\text{if } s(y,z)\geq 1.
 \end{array} \right.
\end{align}
As
$$f[x\mid y\oplus z]=\left\{\begin{array}{ll}0,&\text{if } s(x,y\oplus z)=0,\\
((x\oplus y\oplus z)\text{-}m)_*\Big(\sum\limits_{i=0}^{s(x,y\oplus z)\text{-}1} \wp(iq)_* v_1\Big),&\text{if } s(x,y\oplus z)\geq 1,
 \end{array} \right.
$$
one concludes the formula
$$x_*f[ y\mid z]+f[x\mid y\oplus z]=\left\{\begin{array}{ll}0,&\hspace{-1cm}\text{if } s(y,z)=0=s(x,y\oplus z),\\
((x\oplus y\oplus z)\text{-}m)_*\Big(\sum\limits_{i=0}^{s(y,z)+s(x,y\oplus  z)\text{-}1}\wp (iq)_* v_1\Big),&\text{otherwise}.
 \end{array} \right.
$$

Similarly, one sees that
$$z_*f[x\mid y]+f[x\oplus y\mid z]=\left\{\begin{array}{ll}0,&\hspace{-1cm}\text{if } s(x,y)=0=s(x\oplus y,z),\\
((x\oplus y\oplus z)\text{-}m)_*\Big(\sum\limits_{i=0}^{s(x,y)+s(x\oplus y, z)\text{-}1}\wp (iq)_* v_1\Big),&\text{otherwise},
 \end{array} \right.
$$
and the equality in \eqref{f3-cocy} follows by comparison using \eqref{sisco}.

Finally,  \underline{for a cell}
$[x\mid y\mid z\mid t\mid \cdots]=[x\mid y\mid z\mid \tau]$
 \underline{of dimension higher than 3} we use the formulas
\begin{equation}\label{fpch}
\partial[a\mid x\mid b]=[\partial[a\mid x]\mid b]+[a\mid \partial[x\mid b]]
\end{equation}
which holds for any even chain $a$ and any other chain $b$ of $\B(\ZZ C)$, and
\begin{equation}\label{forpoint}
\partial(c\bullet d)=c\bullet \partial d,
\end{equation}
which holds for any chains $c,d\in \mathcal{R}$. Thus, as we know that $f\partial[x\mid y\mid z]=0$, induction gives
\begin{equation*}
\begin{split}
 f\partial[x\mid y\mid z\mid \tau]&\overset{\eqref{fpch}}=  f[\partial[x\mid y\mid z]\mid \tau]+f[x\mid y\mid \partial[z\mid \tau]]\\ &=
 f\partial[x\mid y\mid  z]\bullet f[\tau]+f[x\mid y]\bullet f\partial[z\mid \tau]
 =
 f[x\mid y]\bullet \partial f[z\mid \tau]\\&\overset{\eqref{forpoint}}=\partial(f[x\mid y]\bullet f[z\mid \tau])=\partial f[x\mid y\mid z\mid \tau]
\end{split}
 \end{equation*}

{\em Part 3.} Here we show that $f$ preserves products.
It is enough to prove that $f(\sigma \circ \tau)=f(\sigma)\circ f(\tau)$ for  cells $\sigma=[x_1\mid \cdots\mid x_n]$ and $\tau=[y_1\mid \cdots\mid y_{n'}]$ of $\B(\ZZ C)$.

As in \cite[page 99]{E-M-II}, a term $T=\pm[t_1\mid \cdots\mid t_{n+n'}]$ in the shuffle product \eqref{sp1} of $\sigma$ and $\tau$ is called {\em mixed} whenever there exists an index $i$ such that $t_{2i-1}$ is an $x$ of $\sigma$  and $t_{2i}$ an $y$ of $\tau$, or vice versa. Choose the first index $i$ for each
mixed $T$, and let $T' $ be the term obtained from $T$ by interchanging $t_{2i-1}$ with $t_{2i}$.
Since $f[x,y]$ is symmetric,
\begin{align}\nonumber
f(T)&=f[t_1\mid t_2]\bullet \cdots \bullet f[t_{2i-1}\mid t_{2i}]\bullet f[t_{2i+1}\mid \cdots ] \\ \nonumber
&=f[t_1\mid t_2]\bullet \cdots \bullet f[t_{2i}\mid t_{2i-1}]\bullet f[t_{2i+1}\mid \cdots ]=f(T').
\end{align}
Since $T$ and $T'$ have opposite signs, the
results cancel and $f(\sigma\circ \tau)=\sum f(T)$, with summation taken only over the unmixed
terms, and where the sign of each term due the shuffle is always plus. If $n=2r+1$ and $n'=2r'+1$ are both odd,
there are no unmixed terms, so $f(\sigma\circ \tau) =0$ in agreement with the fact that $f(\sigma)\circ f(\tau)=0$ (since $w_k\circ w_l=0$). If $n=2r$ and $n'=2r'$ are both even, the unmixed
terms $T$ are obtained by taking all shuffles of the $r$ pairs $(x_1,x_2),...,(x_{2r-1},x_{2r})$ through
the pairs $(y_1,y_2),...,(y_{2r'-1},y_{2r'})$. For any such a shuffle
$$
f(T)=f[x_1\mid x_2]\bullet \cdots\bullet f[x_{2r-1}\mid x_{2r}]\bullet f[y_1\mid y_2]\bullet \cdots \bullet f[y_{2r'-1},y_{2r'}]=f(\sigma)\bullet f(\tau)
$$
and the number of such shuffles is $\binom{r+r'}{r}$, hence
$$
f(\sigma\circ \tau)={\scriptsize \begin{pmatrix}r+r'\\r  \end{pmatrix}}f(\sigma)\bullet f(\tau)=f(\sigma)\circ f(\tau),
$$
as desired. For $n=2r$ and  $n'=2r'+1$, the unmixed terms $T$ are as above but with the last argument $y_{2r'+1}$
always at the end. Hence, for each of them
$$
f(T)=f[x_1\mid x_2]\bullet \cdots\bullet f[x_{2r-1}\mid x_{2r}]\bullet f[y_1\mid y_2]\bullet \cdots \bullet f[y_{2r'-1},y_{2r'}]\bullet f[y_{2r'+1}]=f(\sigma)\bullet f(\tau),
$$
and therefore $f(\sigma\circ \tau)={\scriptsize \begin{pmatrix}r+r'\\r \end{pmatrix}}f(\sigma)\bullet f(\tau)=f(\sigma)\circ f(\tau)$. The remaining case $n=2r+1$ and $n'=2r'$ is treated similarly.
\end{proof}

 \begin{proposition}\label{gep}   A  morphism $g:\mathcal{R}\to \B(\ZZ C)$, of DGA-algebras over $\HH C $,
 may be defined by the recursive formulas
\begin{equation}\label{defg}\left\{\begin{array}{ccl}gv_0&=&[\ ],\\[4pt] gw_k&=&[gv_k\mid 1],  \\[4pt]
 gv_{k+1}&=&\sum\limits_{t<m+q}(m+q-t-1)_*[gw_k\mid t]-\sum\limits_{s<m}(m-s-1)_*[gw_k\mid s].
 \end{array}\right.
\end{equation}
\end{proposition}
\begin{proof}
\noindent {\em Part 1.} We show here that the assignment in \eqref{defg} extends to a morphism of complexes of
$\HH C $-modules.
By Proposition \ref{adfu}, we have to verify that
$g v_k\in \B(\ZZ C)_{2k}(\wp(km))$  and   $g w_k\in \B(\ZZ C)_{2k+1}(\wp(km+1))$. Clearly $gv_0=[\ ] \in \B(\ZZ C)_{0}(0)$.
Assume that $g v_k\in \B(\ZZ C)_{2k}(\wp(km))$. Then, we have
$$
gw_k=[gv_k\mid 1]\in \B(\ZZ C)_{2k+1}(\wp(km)\oplus 1)=\B(\ZZ C)_{2k+1}(\wp(km+ 1)),
$$
as required. Moreover, for any $t<m+q$ and $s<m$,
$$
(m+q-t-1)_*[gw_k\mid t], (m-s-1)_*[gw_k\mid s]\in  \B(\ZZ C)_{2k+2}(\wp((k+1)m)),
$$
since
$$
(m+q-t-1)\oplus \wp(km+1)\oplus t=\wp((k+1)m)= (m-s-1)\oplus \wp(km+1)\oplus s.
$$
Whence $gv_{k+1}\in \B(\ZZ C)_{2k+2}(\wp((k+1)m))$.

\vspace{0.2cm}

\noindent {\em Part 2.} Here we  shall prove, as an auxiliary result, that
\begin{equation}\label{auxg}
gv_k\circ [1]=gw_k,\ \ gw_k\circ [1]=0,
\end{equation}
where $\circ=\circ_1$ is the shuffle product \eqref{sp1} of $\B(\ZZ C)$.
Clearly $gv_0\circ [1]=[\ ]\circ [1]=[1]=[gv_0\mid 1]=gw_0$. Assuming the result for $gv_k$, we have
$$
gw_k\circ [1]=gv_k\circ [1]\circ [1]= gv_k\circ ([1\mid 1]-[1\mid 1])=0,
$$
from where, in addition, it follows that, for any $t\in C$,
$$
[gw_k\mid t]\circ [1]=[gw_k\mid t\mid 1]-[gw_k\circ [1]\mid t]=[gw_k\mid t\mid 1],
$$
whence
\begin{align}\nonumber
gv_{k+1}\circ [1]&=\sum\limits_{t<m+q}(m+q-t-1)_*[gw_k\mid t\mid 1]-\sum\limits_{s<m}(m-s-1)_*[gw_k\mid s\mid 1]
\\ \nonumber
&=[gv_{k+1}\mid 1]=gw_{k+1}.
\end{align}

\noindent {\em Part 3.} We now prove recursively that $\partial g=g\partial$.

For argument $w_0$ is immediate: $\partial gw_0=\partial [1]=0$.
For argument $v_{k+1}$, first observe that  $\partial gw_k=0$  gives, for ant $t\in C$,
\begin{align}\nonumber
\partial[gw_k\mid t]&=\partial[gv_k\mid 1 \mid t]\overset{\eqref{fpch}}=[\partial [gv_k\mid 1]\mid t]+[gv_k\mid \partial[1\mid t]]
\\ \nonumber &=[\partial gw_k\mid t]+[gv_k\mid \partial[1\mid t]]= [gv_k\mid \partial[1\mid t]]\\ \nonumber &=
1_*[gv_k\mid t]-[gv_k\mid 1\oplus t]+t_*[gv_k\mid 1]\\ \nonumber &=
1_*[gv_k\mid t]-[gv_k\mid 1\oplus t]+t_*gw_k.
\end{align}
Then,
\begin{align}\nonumber \partial gv_{k+1}&=\sum\limits_{t<m+q}(m+q-t-1)_*\partial[gw_k\mid t]-\sum\limits_{t<m}(m-t-1)_*\partial[gw_k\mid t]
 \\
\nonumber
&=\sum\limits_{t<m+q-1}(m+q-t)_*[gv_k\mid t]-(m+q-t-1)_*[gv_k\mid 1+t]+ (m+q-1)_*gw_k\\ \nonumber
&\hspace{0.4cm}+1_*[gv_k\mid m+q-1]-[gv_k\mid m]+(m+q-1)_*gw_k\\ \nonumber
&\hspace{0.4cm} -\sum\limits_{t<m}(m-t)_*[gv_k\mid t]-(m-t-1)_*[gv_k\mid 1+ t]+(m-1)_*gw_k \\ \nonumber
&= -1_*[gv_k\mid m+q-1] +(m+q-1)\big((m+q-1)_*gw_k\big)\\ \nonumber &\hspace{0.4cm}+1_*[gv_k\mid m+q-1]-[gv_k\mid m]+(m+q-1)_*gw_k+[gv_k\mid m]-m(m-1)_*gw_k\\ \nonumber
&= (m+q)\big((m+q-1)_*gw_k\big)-m\big((m-1)_*gw_k\big) =g\partial v_{k+1}.
\end{align}

And for argument $w_{k+1}$,

\begin{align}\nonumber \partial g w_{k+1}&\overset{\eqref{auxg},\eqref{dga3}}=
\partial gv_{k+1}\circ [1]
=\Big((m+q)\big((m+q-1)_*gw_k\big)-m\big((m-1)_*(gw_k)\big)\Big)\circ [1]\overset{\eqref{auxg}}=0.
\end{align}
\noindent {\em Part 4.} Here we show that $g$ preserves products by proving
that $g(a\circ b)=ga\circ gb$ for  $a,b\in \{v_k,w_l\}$. For the case when $a=w_k$
 and $b=w_l$, we have
$$gw_k\circ gw_l\overset{\eqref{auxg}}=gv_k\circ [1]\circ gw_l\overset{\eqref{auxg}}=0=g(w_k\circ w_l).$$
To prove the remaining cases, first observe that if $gv_k\circ gv_l=g(v_k\circ v_l)$ for some $k$ and $l$, then
\begin{align*}\nonumber
gw_k\circ gv_l&=gv_k\circ [1]\circ gv_l =gv_k\circ gv_l\circ [1] =g(v_k\circ v_l)\circ [1]=
\\&=
{\scriptsize \begin{pmatrix} k+l\\ k \end{pmatrix}}gv_{k+l}\circ [1]= {\scriptsize \begin{pmatrix} k+l\\ k
\end{pmatrix}} gw_{k+l}=g(w_k\circ v_l).
\end{align*}
Next, we show that $gv_k\circ gv_l=g(v_k\circ v_l)$ by induction. The case when $k=0$ or $l=0$ is immediate, since $gv_0=[\ ]$. Now, using that, for any $t,s\in C$,
$$
[gw_k\mid t]\circ [gw_l\mid s]\overset{\eqref{forshufpro}}=[[gw_k\mid t]\circ gw_l\mid s]+
 [gw_k\circ [gw_l\mid s],t],
$$
  we have
\begin{align}\nonumber
gv_{k+1}\circ gv_{l+1}&=\sum\limits_{s<m+q}(m+q-s-1)_*\Big[\sum\limits_{t<m+q}(m+q-t-1)_*[gw_k\mid t]\circ gw_l\mid s\Big]
\\ \nonumber &\hspace{0.4cm}-\sum\limits_{s<m+q}(m+q-s-1)_*\Big[\sum\limits_{t<m}(m-t-1)_*[gw_k\mid t]\circ gw_l\mid s\Big]\\ \nonumber &\hspace{0.4cm}
+ \sum\limits_{t<m+q}(m+q-t-1)_*\Big[gw_k\circ\sum\limits_{s<m+q}(m+q-s-1)_*[gw_l\mid s]\mid t\Big]
\\  \nonumber &\hspace{0.4cm}
- \sum\limits_{t<m+q}(m+q-t-1)_*\Big[gw_k\circ\sum\limits_{s<m}(m-s-1)_*[gw_l\mid s]\mid t\Big]\\ \nonumber &\hspace{0.4cm}
- \sum\limits_{t<m}(m-t-1)_*\Big[gw_k\circ\sum\limits_{s<m+q}(m+q-s-1)_*[gw_l\mid s]\mid t\Big]\\ \nonumber &\hspace{0.4cm}
 +\sum\limits_{t<m}(m-t-1)_*\Big[gw_k\circ\sum\limits_{s<m}(m-s-1)_*[gw_l\mid s]\mid t\Big]\\
 \nonumber &\hspace{0.4cm}-\sum\limits_{s<m}(m-s-1)_*\Big[\sum\limits_{t<m+q}(m+q-t-1)_*[gw_k\mid t]\circ gw_l\mid s\Big]
\\
 \nonumber &\hspace{0.4cm}+\sum\limits_{s<m}(m-s-1)_*\Big[\sum\limits_{t<m}(m-t-1)_*[gw_k\mid t]\circ gw_l\mid s\Big],
 \end{align}
 and then, by induction,

 \begin{align}\nonumber
  gv_{k+1}&\circ gv_{l+1}=\\ \nonumber
&= \sum\limits_{s<m+q}(m+q-s-1)_*[gv_{k+1}\circ gw_l\mid s] + \sum\limits_{t<m+q}(m+q-t-1)_*[gw_k\circ gw_{l+1}\mid t]
  \\ \nonumber &\hspace{0.4cm}
 -\sum\limits_{s<m}(m-s-1)_*[gv_{k+1}\circ gw_l\mid s] - \sum\limits_{t<m}(m-t-1)_*[gw_k\circ gw_{l+1}\mid t]
  \\ \nonumber &=\hspace{0.4cm}
 {\scriptsize \begin{pmatrix}k+l+1\\k+1\end{pmatrix}}\Big(\sum\limits_{s<m+q}(m+q-s-1)_*[gw_{k+l+1}\mid s]-\sum\limits_{s<m}(m-s-1)_*[gw_{k+l+1}\mid s]\Big) \\ \nonumber &\hspace{0.4cm}+
 {\scriptsize \begin{pmatrix}k+l+1\\k\end{pmatrix}}\Big( \sum\limits_{t<m+q}(m+q-t-1)_*[gw_{k+l+1}\mid t]
 - \sum\limits_{t<m}(m-t-1)_*[gw_{k+l+1}\mid t]\Big)
  \\ \nonumber &=\hspace{0.4cm}
{\scriptsize   \begin{pmatrix}k+l+1\\k+1\end{pmatrix} } gv_{k+l+2}+{\scriptsize \begin{pmatrix}k+l+1\\k\end{pmatrix}}
  gv_{k+l+2}= {\scriptsize \begin{pmatrix}k+l+2\\k+1\end{pmatrix}} gv_{k+l+2}\\[5pt] \nonumber
  &=\hspace{0.4cm} g(v_{k+1}\circ v_{l+1}).
\end{align}
\end{proof}
Now, we are ready to establish the following key result.
\begin{theorem}\label{amtheo1}
 The morphisms $f:\B(\ZZ C)\to \mathcal{R}$ and $g:\mathcal{R}\to\B(\ZZ C)$, as defined above, form a contraction.
\end{theorem}
\begin{proof}
{\em Part 1.} We start by showing that the  composite $fg$ is the identity. Clearly $fgv_0=f[\ ]=v_0$. Then, induction gives
\begin{equation*}
 \begin{split}
 fg w_k&\overset{\eqref{auxg}}=f(gv_k\circ [1])=fgv_k\circ f[1]=v_k\circ w_0=w_k,\\
  fgv_{k+1}&= \sum_{t<m+q}(m+q-t-1)_*f[gv_k|1\mid t]-\sum_{s<m}(m-s-1)_*f[gv_k|1|s]
   \\
  &=
  \sum_{t<m+q}(m+q-t-1)_*(f[gv_k]\bullet f[1\mid t])-\sum_{s<m}(m-s-1)_*(f[gv_k]\bullet f[1|s])
  \\
  &=
  v_k\bullet f[1|m+q-1]=v_k\bullet v_1=v_{k+1}.
 \end{split}
\end{equation*}

\noindent {\em Part 2.} Here, we describe the composite $gf$. Clearly $gf[\ ]=[\ ]$ and
  $gf[x]=x((x-1)_*[1])$. For those 2-cells $[x\mid y]$ such that $x+y<m+q$ we have  $gf[x\mid y]=0$, and, as we prove below,
  the effect of $gf$ on the 2-cells $[x\mid y]$ with $x+y\geq m+q$ is described by the formula
  \begin{align} \label{forgf}
gf[x\mid y]&=
 \sum_{t=x+y-m-q}^{m+q-1}(x+y-t-1)_*[1\mid t]+
\sum_{t=0}^{r-1}(m+r-t-1)_*[1\mid t]\\ \nonumber &\hspace{0.4cm}
-\sum_{t=0}^{m-1}(m+r-t-1))_*[1\mid t]+ \sum_{i=1}^{s(x,y)-1}
\sum_{t=(i-1)q+r}^{iq+r-1} (m+iq+r-t-1)_*[1\mid t]\\ \nonumber&\hspace{0.4cm}
+
\sum_{i=1}^{s(x,y)-1}\sum_{t=m}^{m+q-1} (m+iq+r-t-1)_*[1\mid t],
\end{align}
where we write $x+y=m+s(x,y)q+r$ with $0\leq r<q$ (so that $x\oplus y=m+r$). Concerning the two last terms, note that
$(s(x,y)-1)q+r<m+q$ whenever $s(x,y)\geq 2$, since  $m+s(x,y)q+r=x+y<2m+2q$.

In effect, by definition of $f$ and $g$,  we have
$$
gf[x\mid y]=\sum_{i=0}^{s(x,y)-1}\Big(\sum_{t=0}^{m+q-1} \wp(m+(i+1)q+r-t-1)_*[1\mid t] -
\sum_{t=0}^{m-1} \wp(m+iq+r-t-1)_*[1\mid t] \Big).
$$
Then, since for any $i\geq 1$ and $t<r$ is $\wp\big(m+(i+1)q+r-t-1\big)=\wp(m+iq+r-t-1)$, we see that
\begin{align*}
gf[x\mid y]&=\sum_{t=r}^{m+q-1}(m+q+r-t-1)_*[1\mid t]+\sum_{t=0}^{r-1}(m+r-t-1)_*[1\mid t]-\sum_{t=0}^{m-1}(m+r-t-1)_*[1\mid t]\\ &+
\sum_{i=1}^{s(x,y)-1}\Big(\sum_{t=r}^{m+q-1} \wp(m+(i+1)q+r-t-1)_*[1\mid t] -
\sum_{t=r}^{m-1} \wp(m+iq+r-t-1)_*[1\mid t] \Big)\\ &=
\sum_{t=0}^{r-1}(m+r-t-1)_*[1\mid t]-\sum_{t=0}^{m-1}(m+r-t-1)_*[1\mid t]
\\ &+
\sum_{i=0}^{s(x,y)-1}\sum_{t=r}^{m+q-1}\wp(m+(i+1)q+r-t-1)_*[1\mid t] -
\sum_{i=1}^{s(x,y)-1}\sum_{t=r}^{m-1} \wp(m+iq+r-t-1)_*[1\mid t],
\end{align*}
from where \eqref{forgf} follows thanks to the equalities
\begin{align*}
\sum_{t=r}^{m+q-1}\wp(m+(i+1)q+r-t-1)_*[1\mid t]=&
\sum_{l=0}^{i-1}\sum_{t=lp+r}^{(l+1)q+r-1}(m+(l+1)q+r-t-1)_*[1\mid t]\\&+\sum_{t=iq+r}^{m+q-1}(m+(i+1)q+r-t-1)_*[1\mid t],
\\
\sum_{t=r}^{m-1}\wp(m+iq+r-t-1)_*[1\mid t]=&
\sum_{l=1}^{i-1}\sum_{t=(l-1)q+r}^{lq+r-1}(m+lq+r-t-1)_*[1\mid t]\\&+\sum_{t=(i-1)q+r}^{m-1}(m+iq+r-t-1)_*[1\mid t].
\end{align*}

Finally,  to complete the description of the composite $gf$, for generic cells $[x\mid y\mid \sigma]$ of dimensions greater than 2 we have the formula
 \begin{equation}\label{forgf2}
 gf[x\mid y\mid \sigma]=[gf[x,y]\mid gf[\sigma]].
 \end{equation}
 In effect, as $gf[x\mid y\mid \sigma]=g(f[x,y]\bullet f[\sigma])$, by linearity, it suffices to observe that, for any $k\geq 1$,
$$
\begin{array}{cc}g(v_1\bullet w_k)=[gv_1\mid gw_k],& g(v_{1}\bullet v_k)=[gv_1\mid gv_k],\end{array}
$$
or, equivalently, that  $gw_{k+1}=[gv_1\mid gw_k]$ and $gv_{k+1}=[gv_1\mid gv_k]$.
But these last equations are immediate for $k=1$, and for higher $k$ by a straightforward induction.
\begin{comment}
\begin{equation*}
 \begin{split}
 gv_{k+1}&=\sum_{t<m+q}(m+q-1)_*[gw_k\mid t]-\sum_{s<m} (m-s-1)_*[gw_k\mid s]
  \\
  &=
  \sum_{t<m+q}(m+q-1)_*[gv_1\mid gw_{k-1}\mid t]-\sum_{s<m} (m-s-1)_*[gv_1\mid gw_{k-1}\mid s]=[gv_1\mid gv_k],\\
   gw_{k+1}&=[gv_{k+1}\mid 1]=[gv_1\mid gv_k\mid 1]=[gv_1\mid gw_k].
  \end{split}
\end{equation*}
\end{comment}

\noindent {\em Part 3.} We establish here a homotopy $\Phi$ from $gf$ to the identity, which is determined by the recursive formulas
\begin{equation}\label{defphi}\left\{
 \begin{array}{l}
 \Phi[\ ]=0,\\[4pt]
  \Phi[x]=\sum\limits_{t<x}(x-t-1)_*[1\mid t],
   \\
  \Phi[x\mid y\mid \sigma]=[\Phi[x]\mid y\mid \sigma]+[gf[x\mid y]\mid \Phi[\sigma]].
 \end{array}\right.
\end{equation}
Since, for any $t<x$ in $C$, $(x-t-1)\oplus 1\oplus t=x$, we see that $\pi\Phi[x]=x$ and then, by recursion, that $\pi\Phi[x\mid y\mid \sigma]=x\oplus y\oplus \pi[\sigma]$. Hence, by Proposition \ref{adfu}, the formulas above determine an endomorphism of the complex of $\HH C$-modules $\B(\ZZ C)$, which is of differential degree $+1$.

Next, we prove that $\Phi:gf\Rightarrow id$ is actually a homotopy:

\vspace{0.2cm}
\underline{For a $1$-cell} $[x]$ is $\Phi\partial[x]=0$, and
\begin{align}\nonumber
\partial \Phi[x]&= \sum_{t<x}(x-t)_*[t]-(x-t-1)_*[1+t]+(x-1)_*[1]
=-[x]+x((x-1)_*[1])=-[x]+gf[x],
\end{align}
as required.

\vspace{0.2cm}
\underline{For a $2-$cell} $[x\mid y]$ we have
\begin{equation*}
  \begin{split}
   &(\partial\Phi+\Phi \partial)[x\mid y]=
   \sum_{t<x} (x-t-1)_*(1_*[t\mid y]-[1+ t,y] +[1\mid t\oplus y]-y_*[1\mid t])
   \\
  &\hspace{0.8cm}
  +\sum_{t<y}(x\oplus(y-t-1))_*[1\mid t]-\sum_{t<x\oplus y}((x\oplus y)-t-1)_*[1\mid t]
  +\sum_{t<x}((x-t-1)\oplus y)_*[1\mid t]
  \\
  &
  = \sum_{t<x}(x-t)_*[t\mid y]-(x-t-1)_*[1+t\mid y] +\sum_{t<x}(x-t-1)_*[1\mid t\oplus y]
  -\sum_{t<x} ((x-t-1)\oplus y)_*[1\mid t]
  \\
  &\hspace{0.8cm}
   +\sum_{t<y}(x\oplus(y-t-1))_*[1\mid t]-\sum_{t<x\oplus y}((x\oplus y)-t-1)_*[1\mid t]
  +\sum_{t<x}((x-t-1)\oplus y)_*[1\mid t]
  \\
  &=
  -[x,y] + \sum_{t<x}(x-t-1)_*[1\mid t\oplus y] +
   \sum_{t<y}(x\oplus(y-t-1))_*[1\mid t]
    -\sum_{t<x\oplus y}((x\oplus y)-t-1)_*[1\mid t].
  \end{split}
 \end{equation*}

\underline{If $s(x,y)=0$} then, for any $t<x$,  $t\oplus y=t+y$ and $x\oplus(y-t-1)=x+y-t-1=(x\oplus y)-t-1$. Therefore
$$\sum_{t<x}(x-t-1)_*[1\mid t+ y] +
   \sum_{t<y}(x+y-t-1)_*[1\mid t]
    -\sum_{t<x+ y}(x+ y-t-1)_*[1\mid t]=0,
$$
and, since $gf[x\mid y]=0$, it follows that $(\partial\Phi+\Phi \partial)[x\mid y]=-[x\mid y]+gf[x\mid y]$, as required.

\vspace{0.2cm}
\underline{If $s(x,y)>0$},
 the composite $gf[x\mid y]$ has been computed in \eqref{forgf} and,
 writing as there $x+y=m+s(x,y)q+r$ with $0\leq r<q$, we have
\begin{align*}
\sum\limits_{t<x}(x-t-1)_*&[1\mid t\oplus y]=\sum_{l=1}^{s(x,y)-1}\underset{m+lq\leq t+y<m+(l+1)q}{\sum_{t<x}}\hspace{-0.8cm}(x-t-1)_*[1\mid t+y-lq]\\
&\hspace{0.4cm} + \underset{t+y<m+q}{\sum_{t<x}}(x-t-1)_*[1\mid t+y]
+\underset{m+s(x,y)q\leq t+y}{\sum_{t<x}}\hspace{-0.8cm}(x-t-1)_*[1\mid t+y-s(x,y)q].
\end{align*}
Now, making the changes $u=t+y-lq$, $u=t+y$, and $u=t+y-s(x,y)q$ in the respective terms, and then renaming the $u$ again by $t$,   we obtain
\begin{equation*}
 \begin{split}
\sum\limits_{t<x}(x-t-1)_*[1\mid t\oplus y]&=
\sum_{i=1}^{s(x,y)-1}\sum_{t=m}^{m+q-1}(m+iq+r-t-1)_*[1\mid t]
\\
&\hspace{0.4cm}+ \sum_{t=y}^{m+q-1} (x+y-t-1)_*[1\mid t]+
\sum_{t=m}^{m+r-1}(m+r-t-1)_*[1\mid t].
 \end{split}
\end{equation*}

Similarly, we have
\begin{equation*}
 \begin{split}
  \sum_{t<y}(x\oplus(y-t-1))_*[1\mid t]= &
   \sum_{i=1}^{s(x,y)-1} \sum\limits_{t=(i-1)q+r}^{iq+r-1}(m+iq+r-t-1)_*[1\mid t]
  \\
  & +\sum\limits_{t=x+y-m-q}^{y-1} (x+y-t-1)_*[1\mid t]
 +\sum_{t=0}^{r-1} (m+r-t-1)_*[1\mid t],
 \end{split}
\end{equation*}
and $$
\sum\limits_{t<x\oplus y}((x\oplus y)-t-1)_*[1\mid t]=\sum\limits_{t=0}^{m+r-1}(m+r-t-1)_*[1\mid t].$$ Hence, a direct comparison with \eqref{forgf} gives that $(\partial\Phi+\Phi \partial)[x\mid y]=-[x\mid y]+gf[x\mid y]$, as required.

\vspace{0.2cm} Finally, we prove that $(\partial\Phi+\Phi\partial)(\tau)=-\tau+gf(\tau)$ if \underline{$\tau$ is a cell of dimension $3$ or greater}.  To do so, previously observe  that, for any generic cell $\gamma$ of $\B(\ZZ C)$, we have
\begin{equation}\label{aufpargf}
\partial[gf[x\mid y]\mid \Phi(\gamma)]=[gf[x\mid y]\mid \partial \Phi(\gamma)].
\end{equation}
 To prove it, by linearity,  it suffices to check that $ \partial[gv_1\mid 1\mid  \beta]=[gv_1\mid \partial[1\mid \beta ]]$, for any generic cell $\beta$:
\begin{equation*}
 \begin{split}
  \partial[gv_1\mid 1\mid  \beta ]& \overset{\eqref{fpch}}=[\partial[gv_1\mid 1]\mid  \beta ] +[gv_1\mid \partial[1\mid \beta ]]
  \\ &
  \overset{\eqref{defg}}= [\partial gw_1\mid \beta]  +[gv_1\mid \partial[1\mid \beta ]] =[g\partial w_1\mid \beta ]  +[gv_1\mid \partial[1\mid \beta ]]
  = [gv_1\mid \partial[1\mid \beta ]].
 \end{split}
\end{equation*}

Now, according to the definition in \eqref{defphi}, on chains $c$ of $\B(\ZZ C)$ of dimensions 2 or greater, we can write $\Phi(c)=\Phi_1(c)+\Phi_2(c)$, where $\Phi_1$ and $\Phi_2$ are the morphisms of $\HH C $-modules given on generic cells by $\Phi_1[x\mid y\mid \sigma]=[\Phi[x]\mid y\mid \sigma]$ and $\Phi_2[x\mid y\mid \sigma]=[gf[x\mid y]\mid \Phi(\sigma)]$. Then, for the generic cell $\tau=[x\mid y\mid z\mid \rho]$, as
$$\partial\tau=[\partial[x\mid y]\mid z\mid \rho]-[x\mid \partial[y\mid z\mid \rho]]=
[\partial[x\mid y\mid z]\mid \rho]+[x\mid y\mid \partial[z\mid \rho]],$$
we have
\begin{align*}
\Phi\partial(\tau)&=\Phi_1[\partial[x\mid y]\mid z\mid \rho]-\Phi_1[x\mid \partial[y\mid z\mid \rho]]+ \Phi_2[\partial[x\mid y\mid z]\mid \rho]+\Phi_2[x\mid y\mid \partial[z\mid \rho]]\\ &= [\Phi\partial[x\mid y]\mid z\mid \rho]-[\Phi[x]\mid \partial[y\mid z\mid \rho]]
+[gf\partial[x\mid y\mid z]\mid \Phi[\rho]]+[gf[x\mid y]\mid \Phi\partial[z\mid \rho]]\\&=
[\Phi\partial[x\mid y]\mid z\mid \rho]-[\Phi[x]\mid \partial[y\mid z\mid \rho]]
+[gf[x\mid y]\mid \Phi\partial[z\mid \rho]],
\end{align*}
since $f\partial[x\mid y\mid z]=0$ by \eqref{f3-cocy}. Furthermore, by using \eqref{fpch} and \eqref{aufpargf}, we have
\begin{align*}
\partial\Phi(\tau)&= \partial[\Phi[x]\mid y\mid z\mid \rho]+\partial[gf[x\mid y]\mid \Phi[z\mid \rho]]\\&= [\partial\Phi[x\mid y]\mid z\mid \rho]+[\Phi[x]\mid \partial[y\mid z\mid \rho]]+[gf[x\mid y]\mid \partial\Phi[z\mid \rho]],
\end{align*}
whence, by the already proven above and induction on the dimension of $\rho$, we get
\begin{align*}
(\partial\Phi+\Phi\partial)(\tau)&=[\partial\Phi[x\mid y]\mid z\mid \rho]+[gf[x\mid y]\mid \partial\Phi[z\mid \rho]]+[\Phi\partial[x\mid y]\mid z\mid \rho]+[gf[x\mid y]\mid \Phi\partial[z\mid \rho]]\\&=[(\partial\Phi+\Phi\partial)[x\mid y]\mid z\mid \rho]+
[gf[x\mid y]\mid (\partial\Phi+\Phi\partial)[z\mid \rho]]&\\&=
[-[x\mid y]+gf[x\mid y]\mid z\mid \rho]+[gf[x\mid y]\mid -[z\mid \rho]+gf[z\mid \rho]]\\&=-[x\mid y\mid z\mid \rho]+[gf[x\mid y]\mid gf[z\mid \rho]]\overset{\eqref{forgf2}}=-\tau+gf(\tau),
\end{align*}
as required.

This completes the proof of Theorem \ref{amtheo1}, since the conditions in \eqref{contr2} are
easily verified.
\end{proof}

If $\A$ is any $\HH C $-module, by Proposition \ref{h1l},
the first level cohomology groups
$H^n(C,1;\A)$ are precisely Leech cohomology groups $H^n_{^{_\mathrm{L}}}(C,\A)$. Hence, by Theorem \ref{amtheo1}, these can be computed as $H^n_{^{_\mathrm{L}}}(C,\A)= H^n\text{Hom}_{\HH C}(\mathcal{R},\A)$. Since, by Proposition \ref{adfu}, there are natural isomorphisms
$$
\text{Hom}_{\HH C}(\mathcal{R}_{2k},\A)\cong \A(\wp(km)), \hspace{0.3cm}
\text{Hom}_{\HH C}(\mathcal{R}_{2k+1},\A)\cong \A(\wp(km+1)).
$$
we obtain the following already known result (see \cite[Theorem 5.1]{c-c-2} for a general result computing Leech cohomology groups for finite cyclic monoids).
\begin{proposition}[\cite{c-c-2}, Corollary 5.6] \label{nl11}
Let $C=C_{m,q}$ be the cyclic monoid of index $m$ and period $q$. Then, for any $\HH C$-module $\A$ and any integer $k\geq 0$, there is a natural exact sequence of abelian groups
\begin{equation*}
0\to H^{2k+1}_{^{_\mathrm{L}}}(C,\A)\longrightarrow \A(\wp(km+1))\overset{\partial}\longrightarrow \A(\wp(km+m))
\longrightarrow H^{2k+2}_{^{_\mathrm{L}}}(C,\A)\to 0,
\end{equation*}
where $\partial$ is given by
$
\partial(a) = (m+q)\big((m+q-1)_*a\big)-m\big((m-1)_*a\big)
$.
\end{proposition}

Thus, for instance, if $A$ is any abelian group, regarded as a constant $\HH C$-module, then the homomorphism $\partial:A\to A$ is multiplication by $q$, that is,  $\partial(a)=q\,a$. Therefore, for all $k\geq 0$,
$$
\begin{array}{clc}
H^{2k+1}_{^{_\mathrm{L}}}(C,A)&\cong&\mathrm{Ker}(q:A\to A),\\[5pt]
H^{2k+2}_{^{_\mathrm{L}}}(C,A)&\cong&\mathrm{Coker}(q:A\to A).
\end{array}
$$

We consider now the $r$th level cohomology groups of $C=C_{m,q}$  with $r\geq 2$.
By Theorem \ref{amtheo1} and an iterated use of Lemma \ref{iterate} we conclude that  the complexes of $\HH C$-modules $\B^r(\ZZ C)$ and $\B^{r-1}(\mathcal{R})$ are homotopy equivalent.  Therefore, for any $\HH C$-module $\A$, there are natural isomorphisms
$$
H^n(C,r,\A)\cong H^n\big(\text{Hom}_{\HH C }(\B^{r-1}(\mathcal{R}),\A)\big).
$$
An analysis of the complexes $\B^{r-1}(\mathcal{R})$ tell us that  $\B^{r-1}(\mathcal{R})_n=0$ for $0<n<r$, and that we have the diagram of  suspensions
\begin{equation}\nonumber
 \xymatrix@C=16pt{  &\mathcal{R}_4\ar[r]\ar@{^(->}[d]_{\s}& \mathcal{R}_3
 \ar[r]\ar@{^(->}[d]_{\s}&\mathcal{R}_2\ar[r]\ar@{=}[d]_{\s}& \mathcal{R}_1
 \ar@{=}[d]_{\s}\ar[r]& 0\\ &
 \B(\mathcal{R})_5\ar[r]\ar@{^(->}[d]_{\s}& \B(\mathcal{R})_4\ar[r]\ar@{=}[d]_{\s}&
 \B(\mathcal{R})_3\ar[r]\ar@{=}[d]_{\s}&\B(\mathcal{R})_2\ar[r]\ar@{=}[d]_{\s}&0 \\ &
  \B^2(\mathcal{R})_6\ar[r]& \B^2(\mathcal{R})_5\ar[r]&
  \B^2(\mathcal{R})_4\ar[r]&\B^2(\mathcal{R})_3\ar[r]&0 }
\end{equation}
where

\vspace{0.2cm}\begin{quote}
 $\bullet$ {\em $\B(\mathcal{R})_4$ is the free $\HH C $-module on the binary set consisting of the suspension of the $3$-cell $w_1$ of $\mathcal{R}$ and the $4$-cell $$ [w_0|w_0]$$ with $\pi[w_0\mid w_0]=\wp(2)$,  whose differential is
$  \partial([w_0\mid w_0])=w_0\circ w_0=0$,}
\end{quote}

\vspace{0.2cm}\begin{quote}
$\bullet$ {\em $\B(\mathcal{R})_5$ is  the free $\HH C $-module on the set consisting of the suspension of the $4$-cell
 $v_2$  of $\mathcal{R}$ together the $5$-cells
$$[w_0\mid v_1], \  [v_1\mid w_0]$$
 with $\pi[w_0\mid v_1]=m\oplus 1 =\pi[v_1|w_0]$, and
whose differential is}
\begin{align*}
 \partial[w_0\mid v_1]&=w_1-(m+q)\big((m+q-1)_*[w_0\mid w_0]\big)+m\big((m-1)_*[w_0\mid w_0]\big),\\
 \partial[v_1\mid w_0]&=-w_1-(m+q)\big((m+q-1)_*[w_0\mid w_0]\big)+m\big((m-1)_*[w_0\mid w_0]\big).
\end{align*}
\end{quote}

\vspace{0.2cm}\begin{quote}
$\bullet$ {\em $\B^2(\mathcal{R})_6$ is the free $\HH C $-module on the set consisting of the double suspension of the $4$-cell
 $v_2$  of $\mathcal{R}$, the suspension
 of the $5$-cells $[w_0\mid v_1]$ and the
 $[v_1\mid w_0]$ of $\B(\mathcal{R})_5$,  and the
 $6$-cell
$$[w_0\mid\! \mid w_0]$$
with $\pi[w_0\mid \!\mid w_0]=\wp(2)$, whose differential is
$$\partial[w_0\mid \!\mid w_0]=0.
$$
}
\end{quote}

Then, by Proposition \ref{adfu}, there are natural isomorphisms
$$\begin{array}{l}
\text{Hom}_{\HH C }(\B(\mathcal{R})_{2},\A)\cong \A(1),\
\text{Hom}_{\HH C }(\B(\mathcal{R})_{4},\A)\cong \A(m\oplus1)\times \A(\wp(2)),
 \\[6pt]
 \text{Hom}_{\HH C }(\B(\mathcal{R})_{3},\A)\cong \A(m),
 \
\text{Hom}_{\HH C }(\B(\mathcal{R})_{5},\A)\cong \A(\wp(2m))\times \A(m\oplus 1)\times \A(m\oplus 1),\\[6pt]
\text{Hom}_{\HH C }(\B^2(\mathcal{R})_{6},\A)\cong \A(\wp(2m))\times \A(m\oplus 1)\times \A(m\oplus 1)\times \A(\wp (2)).
\end{array}
$$
In these terms the truncated complex $\text{Hom}_{\HH C }(\B(\mathcal{R}),\A)$ is written as
\begin{equation}\label{hombc}
0\to \A(1) \overset{\partial^1}\to \A(m)\overset{\partial^2}\to \A(m\oplus 1)\times\A(\wp(2))
\overset{\partial^3} \to \A(\wp(2m))\times \A(m\oplus 1)\times \A(m\oplus 1),
\end{equation}
where the coboundaries are given by
$$
\partial^1(a) = -(m+q)\big((m+q-1)_*a\big)+m\big((m-1)_*a\big),
$$
$\partial^2=0$ is the morphism zero, and
\begin{align*}
\partial^3(a,b)=  \Big(-& (m+q)\big((m+q-1)_*a)+m\big((m-1)_*a),\
\\
  a& -(m+q)\big((m+q-1)_*b\big)+m\big((m-1)_*b),
\\
 - & a-(m+q)\big((m+q-1)_*b)+m\big((m-1)_*b\big)\Big),
\end{align*}
while the truncated complex $\text{Hom}_{\HH C }(\B^2(\mathcal{R}),\A)$ is written as
\begin{equation}\label{hombc2}
0\to \A(1) \overset{\partial^1}\to \A(m)\overset{\partial^2}\to \A(m\oplus 1)\times\A(\wp(2))
\overset{\partial^3} \to \A(\wp(2m))\times \A(m\oplus 1)\times \A(m\oplus 1)\times \A(\wp (2)),
\end{equation}
where $\partial^1$ and $\partial^2$ are the same as above whereas
$\partial^3$ acts by
\begin{align*}
\partial^3(a,b)=  \Big(& (m+q)\big((m+q-1)_*a)-m\big((m-1)_*a),\
\\
  &-a  +(m+q)\big((m+q-1)_*b\big)-m\big((m-1)_*b),
\\
 & a+(m+q)\big((m+q-1)_*b)-m\big((m-1)_*b\big),\ 0\Big).
\end{align*}

Then, as an immediate consequence of \eqref{hombc} and  \eqref{hombc2}, we have

\begin{theorem}\label{h23}Let $C=C_{m,q}$ be the cyclic monoid of index $m$ and period $q$. Then, for any $\HH C$-module $\A$, there is a natural exact sequence of abelian groups
$$
0\to H^{2}(C,2;\A)\longrightarrow \A(1)\overset{\partial}\longrightarrow \A(m)
\longrightarrow H^{3}(C,2;\A)\to 0
$$
where $\partial(a) = (m+q)\big((m+q-1)_*a\big)-m\big((m-1)_*a\big)$, and natural isomorphisms
\begin{equation*}
H^4(C,2,\A)\cong H^{5}(C,3;\A)\cong \left\{ b\in \A(\wp(2)) \left|\begin{array}{l} (m+q)^2\wp(2m+q-2)_*b=m^2\wp(2m-2)_*b,\\[4pt]
2(m+q)(m+q-1)_*b=2m(m-1)_*b, \end{array} \right.\right\}.
\end{equation*}

\end{theorem}

Note that in the case when the cyclic monoid is of index $m=1$, the above description of $H^{4}(C,2;\A)$ adopts the simpler form
$$
H^{4}(C,2;\A)\cong \left\{ b\in \A(\wp(2)) \left|\begin{array}{l} (q+1)^2q_*b=b,\\[4pt]
2(q+1)q_*b=2b, \end{array} \right.\right\},
$$
while when $m\geq 2$,
$$
H^{4}(C,2;\A)\cong \left\{ b\in \A(\wp(2)) \left|\begin{array}{l} (2mq+q^2)\wp(2m-2)_*b=b,\\[4pt]
2(m+q)(m+q-1)_*b=2m(m-1)_*b, \end{array} \right.\right\}.
$$

\begin{corollary} For any finite cyclic monoid $C$, any integer $r\geq 1$,  and any $\HH C$-module $\A$, there are natural isomorphisms $$H^{r+1}(C,r;\A)\cong H^2_{^{_\mathrm{L}}}(C,\A)\cong H^2_{^{_\mathrm{G}}}(C,\A).$$
\end{corollary}
\begin{proof}
A direct comparison of the exact sequence in Theorem \ref{h23} with the sequence in Proposition \ref{nl11}, for the case when $k=0$, gives $H^3(C,2;\A)\cong H^2_{^{_\mathrm{L}}}(C,\A)$. Then, the result follows since $H^3(C,2;\A)\cong H^2_{^{_\mathrm{G}}}(C,\A)$ by Proposition \ref{comgc}, and $H^{r+1}(C,r;\A)\cong H^3(C,2;\A)$ by Corollary \ref{c58}.
\end{proof}

\begin{corollary}
For any finite cyclic monoid $C$, any integer $r\geq 2$,  and any $\HH C$-module $\A$, there are natural isomorphisms
$$
H^{r+2}(C,r;\A)\cong H^3_{^{_\mathrm{C}}}(C,\A).
$$
\end{corollary}
\begin{proof}
By Corollary \ref{c59}, $H^{r+2}(C,r;\A)\cong H^5(C,3;\A)$, for any $r\geq 3$. Since, by Theorem \ref{h23}, $H^5(C,3;\A)\cong H^4(C,2,\A)$, the result follows by Proposition \ref{p512}.
\end{proof}

For instance, if $A$ is any abelian group viewed as a constant $\HH C $-module, then
$H^{4}(C,2;A)$ is isomorphic to the subgroup of $A$ consisting of those elements $b$ such that
$$
\left|\begin{array}{l}(m+q)^2b=m^2b,\\ 2qb=0, \end{array}  \right. \Leftrightarrow
\left|\begin{array}{l}(2mq+q^2)b=0,\\ 2qb=0, \end{array}  \right.
\Leftrightarrow
\left|\begin{array}{l}q^2b=0,\\ 2qb=0, \end{array}  \right. \Leftrightarrow (2q,q^2)\,b=0,
$$
where $(2q,q^2)=q(2,q)$ is the greatest common divisor of $2$ and $q$.
This leads to the following isomorphism,
which is analogous to the proven by Eilenberg- Mac Lane for the third abelian cohomology group
of the cyclic group $C_q$ with coefficients in $A$ \cite[\S 21]{E-M-II}.

\begin{corollary} For any finite cyclic monoid $C$, any integer $r\geq 2$,  and any abelian group $A$, there is a natural isomorphism
$$
H^{r+2}(C,r;A)\cong \text{\em Hom}_{\mathbf{Ab}}\big(\mathbb{Z}/(2q,q^2)\mathbb{Z},A\big).
$$
\end{corollary}

\subsection{Cohomology of the infinite cyclic monoid}$~$

In this subsection we focus on the additive monoid of natural numbers $C_\infty=\mathbb{N}$.  As before, we start by introducing a commutative DGA-algebra over $\HH C_\infty$,
$\mathcal{R}$, simpler than $\B(\ZZ C_\infty)$.

For each integer $k=0,1,\dots$, let us choose unitary sets over $ C_\infty$, $\{ w_0\}$ and $\{ v_k\}$, with $\pi w_0=1$ and $\pi v_k=k$. Then,
\begin{equation*}\left\{
\begin{array}{lcl}
\mathcal{R}_{0}&=& \text{the free $\HH C_\infty$-module on  $\{v_0\}$,}\\[4pt]
\mathcal{R}_{1}&=& \text{the free $\HH C_\infty$-module on $\{w_0\}$,} \\ [4pt]
\mathcal{R}_n  &=& 0, \ n\geq 2
\end{array}\right.
\end{equation*}
The differential $\partial=0$ is zero. The augmentation is the canonical isomorphism $\mathcal{R}_0\cong \Z$, and the multiplication
on $\mathcal{R}$ is by determined by the rules $v_0\circ v_0=v_0$, $v_0\circ w_0= w_0$ and $w_0\circ w_0=0$.

 \begin{theorem}\label{t01}
 There are DGA-algebra morphisms $f:\B(\ZZ C_\infty)\to \mathcal{R}$ and $g:\mathcal{R}\to \B(\ZZ C_\infty)$, determined by the formulas
 \begin{equation*}
\left\{\begin{array}{ccl}f[\ ]&=&v_0,\\
f[x]&=&x((x-1)_*w_0)  \end{array}\right.
  \end{equation*}

\begin{equation*}
\left\{\begin{array}{ccl}
        gv_0&=&[\ ],\\[4pt]
gw_0&=&[1],
\end{array}\right.
\end{equation*}
 which form a contraction.
\end{theorem}
\begin{proof}
It is plain to see that above assignments determine well defined morphisms of DGA-algebras over $\HH C_\infty$.
To prove that they form a contraction, we limit ourselves to describe the homotopy $\Phi:gf\Rightarrow id$, by the formula below, because the details are parallel and much more simpler than those in the proof of Theorem \ref{amtheo1}.
 \begin{equation*}\left\{
 \begin{array}{l}
 \Phi[\ ]=0,\\[4pt]
  \Phi[x]=\sum\limits_{0\leq t<x}(x-t-1)_*[1\mid t],
   \\
  \Phi[x\mid \sigma]=[\Phi[x]\mid \sigma],
 \end{array}\right.
\end{equation*}
with $\sigma$ any cell of dimension greater than 1.
\end{proof}

 By Proposition \ref{h1l},  there are isomorphisms $H^n(C_\infty,1;\A)\cong H^n_{^{_\mathrm{L}}}(C_\infty,\A)$, for any $\HH C_\infty$-module $\A$. Then, as consequence of Theorem \ref{t01},
we recover the computation by Leech of the cohomology groups of the monoid $C_\infty$ \cite[Theorem 6.8]{leech}.
\begin{proposition}
For any $\HH C_\infty$-module $\A$,  there are natural isomorphisms
 \begin{equation*}
  H^0_{^{_\mathrm{L}}}(C_\infty,\A)\cong \A(0), \ H^1_{^{_\mathrm{L}}}(C_\infty,\A)\cong \A(1),
 \end{equation*}
and for every $n\geq 2$,
$
H^n_{^{_\mathrm{L}}}(C_\infty,\A)=0.
$
\end{proposition}

 We now pay attention to the second level cohomology groups  of $C_\infty$. By Theorem \ref{t01} and Lemma \ref{iterate},
$
H^n(C_\infty,2;\A)\cong H^n\big(\mathrm{Hom}_{\HH C_\infty}(\B(\mathcal{R}),\A)\big)
$.
An analysis of $\B(\mathcal{R})$ tell us that
\begin{equation*}\left\{
\begin{array}{lcl}
\B(\mathcal{R})_{2k}&=& \text{the free $\HH C_\infty$-module on } \{v_k\},\\[4pt]
\B(\mathcal{R})_{2k+1}&=&0,
\end{array}\right.
\end{equation*}
where, recall,  $\pi v_k=k$; the augmentation is
 the canonical isomorphism $\B(\mathcal{R})_0\cong \Z$ and the product is given by
$$v_k\circ
v_l=\binom{k+l}{k}v_{k+l}.$$

 Hence,
 \begin{proposition}
 For any $\HH C_\infty$-module $\A$, and any integer $k\geq 0$,
 $$
 H^{2k}(C_\infty,2;\A)\cong \A(k), \  \  H^{2k+1}(C_\infty,2;\A)= 0.
 $$
\end{proposition}

From Corollary \ref{c58}, it follows that

\begin{corollary}  For any $\HH C_\infty$-module $\A$, and any integer $r\geq 2$,
$$H^{r+1}(C_\infty,r;\A)=0.$$
\end{corollary}
 We finish by specifying the 3rd level 5-cohomology group of $C_\infty$.

 \begin{proposition} For any $\HH C_\infty$-module $\A$, and any integer $r\geq 3$, there is a natural isomorphism
 $$
 H^{r+2}(C_\infty,r;\A)\cong \{a\in \A(2)\ | \ 2a=0\}.
 $$
 \end{proposition}
\begin{proof}
 By Corollary \ref{c59}, $H^{r+2}(C_\infty,r;\A)\cong H^{5}(C_\infty,3;\A)$.
 An analysis of $\B^2(\mathcal{R})$ tell us that  $\B^2(\mathcal{R})_4=\B(\mathcal{R})_3=0$,
 $\B^2(\mathcal{R})_5=\B(\mathcal{R})_4$ is the free $\HH C_\infty$-module on $\{v_2\}$, where $\pi v_2=2$,  $\B^2(\mathcal{R})_6$ is the free $\HH C_\infty$-module on $\{[v_1\mid\!\mid v_1]\}$, with $\pi[v_1\mid\!\mid v_1]=2$,
 and the differential is
 $$ \partial[v_1\mid\!\mid v_1]=-2v_2.
 $$
Whence, for any $\HH C_\infty$-module $\A$, $H^{5}(C_\infty,3;\A)\cong \{a\in \A(2)\ | \ 2a=0\}$.
 \end{proof}


\begin{thebibliography}{99}

\bibitem{c-c-h} M. Calvo-Cervera, A.M.  Cegarra, B. A. Heredia, Structure and classification of monoidal groupoids, Semigroup Forum 87 (2013) 35-79.


\bibitem{c-c-1} M. Calvo-Cervera, A.M.  Cegarra,  A Cohomology Theory for Commutative Monoids, Mathematics 3 (2015) 1001-1031.

\bibitem{c-c-h2} M. Calvo-Cervera, A.M. Cegarra, B.A. Heredia, On the third cohomology group of a commutative monoid,  Semigroup Forum,  forthcomming paper. Published first on line doi:10.1007/s00233-015-9696-2.

\bibitem{c-c-2} M. Calvo-Cervera, A.M. Cegarra, Computability of the (co)homology of cyclic monoids, preprint available at http://arxiv.org/abs/1602.01272


\bibitem{c-c-3} M. Calvo-Cervera, A.M. Cegarra, Higher cohomologies of commutative monoids: Simplicial representability, preprint.


\bibitem{cegarra2} A.M. Cegarra, E. Khadmaladze, Homotopy classification of braided
graded categorical groups, J. Pure Appl. Algebra 209 (2007) 411-437.

\bibitem{E-M-I} S. Eilenberg, S. MacLane, On the groups $H(\pi,n)$ I, Ann. of Math.  58 (1953) 55-106.

\bibitem{E-M-II} S. Eilenberg, S. MacLane, On the groups $H(\pi,n)$ II. Ann. of Math. 70 (1954) 49-137.

\bibitem{Frobenius} F.G. Frobenius,  \"{U}ber endliche Gruppen, Sitzungsber. Preuss. Akad. Wiss. Berlin (1985) 163-194.

\bibitem{G-Z} P. Gabriel, M. Zisman, Calculus of Fractions and Homotopy Theory,
Springer, Berlin, 1967.

\bibitem{grillet91} P.A. Grillet, Commutative semigroup cohomology. Semigroup Forum 43 (1991) 247–252.

\bibitem{grillet} P.A. Grillet, Commutative Semigroups (Advances in Mathematics), Kluwer Academic Publisher, Philip Drive Norwell, 2001.

\bibitem{hig} P.J. Higgings, Categories and groupoids, Reprints in Theory Appl. Categ.  7, 2005.

\bibitem{j-s} A. Joyal, R. Street, Braided tensor categories, Adv. Math. 82 (1991) 20-78.

\bibitem{kur-Pir} R. Kurdiani, T. Pirashvili, Functor homology and homology of commutative monoids,  Semigroup Forum 92 (2016) 102-120.

\bibitem{leech} J. Leech, Two papers: $\mathcal{H}$-coextensions of monoids and the structure
of a band of groups. Memoirs A.M.S. 157, 1975.

\bibitem{Mac2} S. Mac Lane, Natural associativity and commutativity,
Rice University Studuies 49 (1963) 28-46.

\bibitem{Mac3} S. Mac Lane, Categories for the Working Mathematician, 2nd edition, Graduate Texts in Mathematics 5, Springer, New York, 1998.

\bibitem{maclane} S. Mac Lane,  Homology, Springer Science \& Business Media, 2012.

\bibitem{Saa} N. Saavedra, Cat\'egories tannakiennes, Lect. Notes in Math. 265, Springer-Verlag, 1972.

\bibitem{sinh} H.X. Sinh, Gr-cat\'egories. Th\'ese de Doctorat, Universit\'e Paris VII, 1975.

\bibitem{wells} C. Wells, Extension theories for monoids, Semigroup Forum 16 (1978)
13-35.
\end{thebibliography}
\end{document}